\documentclass[12pt]{amsart}

\usepackage{amssymb,amsthm,amstext,latexsym,amsmath}
\usepackage{dsfont}
\usepackage{enumitem}
\usepackage[mathscr]{eucal}
\usepackage{tensor}
\usepackage[pdftex]{color} 
\usepackage{xcolor}
\usepackage{wrapfig} 
\usepackage[all]{xy}
\usepackage{float}
\usepackage[pdftex]{graphicx} 
\usepackage{longtable}
\usepackage{mathtools}
\usepackage[normalem]{ulem}
\usepackage{cancel}
\usepackage{upgreek}
\usepackage{stmaryrd}
\usepackage{stackrel}

\input xy
\xyoption{all}

\theoremstyle{plain}
\newtheorem{thm}{Theorem}
\newtheorem{Theo}[thm]{Theorem}
\newtheorem*{theor}{Theorem}

\newtheorem{cor}[thm]{Corollary}
\newtheorem{theorem}[thm]{Theorem}
\newtheorem{corollary}[thm]{Corollary}
\newtheorem{lemma}[thm]{Lemma}

\newtheorem{Prop}[thm]{Proposition}
\newtheorem{proposition}[thm]{Proposition}
\newtheorem*{mr}{Main Result}

\newtheorem*{observation}{Observation}
\newtheorem*{coroll}{Corollary}

\newtheoremstyle{rmk}
{9pt}{9pt}{}{}{\bfseries}{}{.5em}{}
\theoremstyle{rmk}
\newtheorem{rmk}[thm]{Remark}

\theoremstyle{alg}

\newtheoremstyle{question}
{9pt}{9pt}{}{}{\bfseries}{}{.5em}{}
\theoremstyle{question}

\numberwithin{equation}{section}
\numberwithin{thm}{section}
\numberwithin{figure}{section}

\theoremstyle{definition}
\newtheorem{definition}[thm]{Definition}

\newtheorem{defn}[thm]{Definition}
\newtheorem{example}[thm]{Example}

\newtheorem{remark}[thm]{Remark}

\newtheorem{Rmk}[thm]{Remark}
\newtheorem{Ex:}[thm]{Example}

\newcommand{\Rr}{\mathbb{R}}            
\newcommand{\Zz}{\mathbb{Z}}            
\newcommand{\XX}{\mathfrak{X}}          
\newcommand{\Lie}{\mathcal{L}}          
\newcommand{\gpd}{G \rightrightarrows M} 

\renewcommand{\H}{H}          
\renewcommand{\gg}{\mathfrak{g}}        

\newcommand{\R}{\mathbb{R}}

\newcommand{\bracket}{\{\cdot,\cdot\}} 
\newcommand{\jm}{(M, L, \{\cdot,\cdot\})} 

\title[Multiplicative Gray Stability]{Multiplicative Gray Stability}

\author[C. Angulo, M. A. Salazar, D. Sepe]{Camilo Angulo \qquad María
  Amelia Salazar \qquad Daniele Sepe}
\email{ca.angulo951@gmail.com}
\email{mariasalazar@id.uff.br}
\email{danielesepe@id.uff.br}

\address{Instituto de Matem\'atica e Estat\'istica, Universidade Federal Fluminense, Niter\'oi, Brazil.}

\date{\today}

\begin{document}

\keywords{Contact groupoids, Jacobi bundles, Poisson manifolds.}
\subjclass{53D35, 22A22, 53D17, 58H15}
\begin{abstract}
  In this paper we prove Gray stability for compact contact groupoids
  and we use it to prove stability results for deformations of the
  induced Jacobi bundles.
\end{abstract}

\maketitle


\section*{Introduction}
Lie groupoids endowed with multiplicative geometries have been used successfully to study singular
geometric structures on their bases (see, e.g.,
\cite{coste,dazord,cf_integrability_poisson,bcwz,cf_geometric,crainic_zhu,crainic_salazar,spencer_operators,PMCT,PMCT2}). Recently,
particular attention has been given to {\bf
  contact groupoids} and the induced Jacobi bundles on their
bases (see, e.g., 
\cite{dazord,kerbrat,crainic_salazar,dual_pairs,vitagliano_dirac,vitagliano_wade}). Contact
geometry is the odd-dimensional analog of symplectic geometry and it
arises naturally in the study of geometric PDEs (see, e.g.,
\cite{arnold_wave}). A contact structure is a maximally non-integrable
hyperplane distribution (see Definition
\ref{defn:contact_structure}). A foundational result in contact
geometry is {\bf Gray stability}: there
are no non-trivial deformations of a contact structure on a compact
manifold (see \cite[Theorem 5.2.1]{gray}). The
main result of this paper is a multiplicative version of Gray
stability, stated informally below (see Theorem \ref{thm:gray_dist}).

\begin{mr}
  A smooth 1-parameter family $\{H_\tau\}$ of multiplicative contact
  structures on a compact Lie groupoid $G$ over a connected manifold
  is trivial, i.e., there exists a smooth 1-parameter
  family $\{\Phi_\tau\}$ of automorphisms of $G$ with $\Phi_0 =
  \mathrm{id}$ such that $d\Phi_\tau(H_\tau) = H_0$ for all
  $\tau$.
\end{mr}

The above should be compared with the very recent
\cite[Theorem 1.2]{cms}, which describes the moduli space of
deformations of multiplicative symplectic forms on a compact
symplectic groupoid.

Combining our Main Result with the stability result for deformations of compact Lie groupoids
proved in \cite{DefLieGpds,Riemstacks}, the following holds (see Definitions \ref{defn:def_contact_gpd} and
\ref{defn:iso_contact}, and Corollary \ref{cor:stability}). 

\begin{coroll}
  Any deformation of a compact contact groupoid over a connected
  manifold is trivial.
\end{coroll}

Motivated by understanding stability phenomena in Poisson geometry
(see \cite{cf_stability}), recently Crainic, Fernandes and
Martínez-Torres have started an ambitious program to study
Poisson manifolds of compact types, i.e., those induced by proper symplectic groupoids. These
enjoy special properties and are, in some sense, rare (see
\cite{PMCT,PMCT2,mt,zwaan}). Contact groupoids induce {\em Jacobi bundles} on
their base manifolds. Jacobi bundles
can be thought of as infinite dimensional Lie algebras of `geometric
type' and include symplectic, contact and Poisson structures as
examples (see, e.g., \cite{kirillov}). Formally, a Jacobi bundle is a line bundle endowed with a local Lie
bracket on its space of sections (see Definition
\ref{defn:jac_str_manifolds}). The present work suggests that
studying Jacobi bundles that are induced by proper contact groupoids
is not only analogous, but also complementary, to the above program. To
illustrate this, we begin with the following remark (see Corollary
\ref{cor:proper_jacobi} for a precise statement). 

\begin{observation}
  A Jacobi bundle induced by a proper co-orientable contact groupoid
  comes from a Poisson bivector.
\end{observation}

We intend to undertake a systematic study of
Jacobi bundles induced by proper contact groupoids in future papers. By the above Observation, a
deformation of a proper co-orientable contact groupoid induces a
deformation of a Jacobi bundle that can be seen as a
deformation of a Poisson bivector. We use
our Main Result to obtain the following characterization of such
deformations (see Theorem
\ref{thm:gray_poisson} for a precise statement).

\begin{theor}
  Let $\{\pi_\tau\}$ be a smooth 1-parameter family of Poisson
  bivectors on $M$ induced by a smooth 1-parameter family of
  co-orientable contact
  structures $\{H_\tau\}$ on a compact Lie groupoid $G$. Then there
  exists a diffeotopy $\{\phi_\tau\}$ of $M$ and a smooth 1-parameter
  family of positive Casimirs $\{a_\tau\}$ of $(M,\pi)$ such that
  $(\phi_\tau)_* \pi_\tau = a_\tau \pi$ for all $\tau$.
\end{theor}

The above Theorem applied to a Lie-Poisson sphere in the dual of
a compact Lie algebra should be compared with \cite[Part (a) of
Theorem 1]{Ionut_spheres}, which is stronger but uses
infinite dimensional techniques (see Remark \ref{rmk:ionut}). \\

Our strategy to prove the Main Result is to consider first the
case in which the contact structures are co-orientable, i.e., given
by the kernel of 1-forms, and then the case in which they are not (see
Theorems \ref{CompactGray} and \ref{thm:main_not_triv}). In
the former, we use the Moser
trick, in analogy with the well-known proof of Gray stability (see,
e.g., \cite[Theorem 2.2.2]{Geiges}, and Theorem \ref{CompactGray} below). We look for the desired smooth 1-parameter of diffeomorphisms $\{\Phi_\tau\}$ of $G$ as the flow of a time-dependent
vector field $X_\tau$. Moreover, since we want $\{\Phi_\tau\}$ to be a family of
Lie groupoid automorphisms, $X_\tau$ must be {\em multiplicative},
i.e., $X_\tau : G \to TG$ need be a Lie groupoid homomorphism for all
$\tau$ (see
Section \ref{sec:contact-groupoids} for the Lie groupoid structure on
$TG$). We achieve this by finding a `good' smooth 1-parameter family
of multiplicative contact forms for $\{H_\tau\}$ (see Definition
\ref{defn:mult_contact_form} and Corollary
\ref{cor:existence_form} for a precise statement). Crucially, this
uses compactness of $G$, which implies that its differentiable
cohomology vanishes in all positive degrees by
\cite[Proposition 1]{crainic}. The above contact forms are multiplicative with
values in a representation of $G$ on the trivial
line bundle that is codified by a Lie groupoid homomorphism $\sigma: G \to
\{\pm 1\}$ (see Definition \ref{defn:mult_form}). This brings in some
small technicalities to prove that $X_\tau$ is multiplicative (see
Appendix \ref{sec:sigma-mult-forms}).

In order to deal with the case
in which the contact structures are not co-orientable, in Section
\ref{subsec:co-or_finite_cover} we introduce a simple, seemingly new
construction for such contact groupoids that we call the {\em
  co-orientable finite cover}. It is analogous to the co-orientable
double cover of a contact manifold (see Remark
\ref{rmk:co-orientable_double_cover}). Given a smooth 1-parameter
family of contact structures on a compact Lie groupoid
$G$, we consider the smooth 1-parameter family of co-orientable
contact structures given by their co-orientable finite covers. This
allows us to argue as above, making sure that the smooth 1-parameter
family of automorphisms of co-orientable finite covers 
descends to $G$ (see the proof of Theorem
\ref{thm:main_not_triv}). Aside from its use in this paper, we expect
that the co-orientable finite cover of contact groupoids will be useful
in the study of multiplicative contact structures that are not
co-orientable.

\subsection*{Outline} In
Section \ref{sec:cont-and-Jac} we discuss properties of contact
groupoids. In Sections \ref{subsec:cont-str}  --
\ref{sec:co-orient-cont}, we recall the basics of contact
structures and of (co-orientable) contact
groupoids. In Section
\ref{subsec:co-or_finite_cover}, we define and study the co-orientable finite
cover of a contact groupoid, which we use in the proof of our main
result. Section \ref{sec:deform-cont-group} formalizes
the notion of `smooth 1-parameter families of contact groupoids'
using deformations as in
\cite{DefLieGpds,Riemstacks} (see also \cite{cms}). In Section \ref{sec:deform-lie-group} we
recall the basics of deformations of Lie groupoids, while in Sections \ref{sec:strict-deform-cont} and
\ref{sec:deform-co-orient} we introduce deformations of
(co-orientable) contact
groupoids. Section \ref{sec:mult-vers-grays} proves our main
result, Theorem \ref{thm:gray_dist}. Sections
\ref{sec:cont-and-Jac} -- \ref{sec:mult-vers-grays} can be read
independently of the remaining sections. In Section
\ref{sec:examples-&-apps} we consider multiplicative Gray stability at
the level of objects. Sections
\ref{subsec:jac-mfds} and \ref{sec:Jacobi_manifolds} introduce Jacobi
bundles and recall how they are induced by contact groupoids. In
Section \ref{sec:at-level-objects-1}, we explain how our main result can be
used to study deformations of Jacobi bundles (see Theorems \ref{thm:gray_Jacobi}
and \ref{thm:gray_poisson}). Three families of examples of compact contact groupoids and their induced
Jacobi bundles are given in Section \ref{sec:three_examples}. Finally,
Appendix \ref{sec:sigma-mult-forms} deals with technical lemmas that
we use in the proof of our main result. 



\subsection*{Acknowledgments} The authors would like to thank Marius
Crainic, David Martínez-Torres and Pedro Frejlich for useful
conversations. C.A. was supported by FAPERJ grant E-26/202.439/2019. M.A.S. and D.S. were partly supported by
CNPq grant Universal 409552/2016-0. M.A.S was partly supported by CNPq grants PQ1 304410/2020-9 and Universal 28/2018, FAPERJ grants JCNE 262012932021 and ARC 262114122019, L'Or\'eal--UNESCO--ABC, and by the Serrapilheira Institute grant Serra-1912-3205. D.S. was partly supported by CNPq grant
PQ1 307029/2018-2 and by FAPERJ grant JCNE E-26/202.913/2019. This study was financed in part by the
Coordenação de Aperfeiçoamento de Pessoal de Nível Superior -- Brazil
(CAPES) -- Finance code 001. 

\subsection*{Conventions}
Throughout the paper, $I \subseteq \R$ is an open interval
containing $0$ and all vector bundles are real. Moreover, if $G_b$ denotes
the fiber of a surjective submersion $G \to B$ over $b \in B$, we
use $i_b : G_b \hookrightarrow G$ for the natural inclusion. This is
used extensively in Sections \ref{sec:deform-cont-group} and
\ref{sec:mult-vers-grays}.

\section{Contact groupoids}\label{sec:cont-and-Jac}
In this section we recall some notions and results regarding 
contact groupoids, most of which are
well-known except for  Proposition
\ref{prop:proper_co-orientable} and Corollary
\ref{cor:proper_co-orientable} in Section \ref{sec:co-orient-cont},
and Section \ref{subsec:co-or_finite_cover}. The main references for 
Lie groupoids are 
\cite{crainic_fernandes_book,mack}, for 
multiplicative distributions is \cite{spencer_operators}, and 
for contact groupoids are 
\cite{crainic_salazar,crainic_zhu,dazord,kerbrat,zamb_zhu}.

\subsection{Contact manifolds}\label{subsec:cont-str}

\begin{defn}\label{defn:contact_structure}
  A {\bf contact structure} on a manifold $N$ is a 
  hyperplane distribution $H \subset TN$ such that the vector 
  bundle morphism $c_H : H \times H \to TN/H$ given at the 
  level of sections by $c_H(X,Y)=[X,Y] \, \mathrm{mod} \, H$ 
  is non-degenerate. The pair $(N,H)$ is a 
  {\bf contact manifold}. 
\end{defn}

There is a dual approach to contact 
structures using 1-forms taking values in line bundles. Given a
contact manifold $(N,H)$, setting $\alpha_{\mathrm{can}} : TN \to L:=
TN/H$, we have that $H = \ker \alpha_{\mathrm{can}}$ and
$\alpha_{\mathrm{can}} \in \Omega^1(N;L)$. Following \cite[Note
$2.3$]{sal_sepe}, we refer to $\alpha_{\mathrm{can}}$ as the 
{\bf generalized contact form} of $(N,H)$. 

Let $(N,H)$ be {\bf co-orientable}, i.e., the 
  line bundle $L \to N$ is trivializable, and let 
$\psi : L \to N \times \mathbb{R}$ be a trivialization. If 
$\alpha_{\mathrm{can}}$ is the generalized contact form of 
$(N,H)$, then $\alpha:= \psi \circ \alpha_{\mathrm{can}} \in
\Omega^1(N)$ and $H = \ker \alpha$.  If $\psi'$ is another
trivialization of $L$, set $\alpha' := \psi' \circ
\alpha_{\mathrm{can}}$. If we identify $\psi' \circ \psi^{-1}$
with a nowhere vanishing $a \in C^{\infty}(N)$, then we have that $\alpha'
= a \alpha_{\psi}$. 

\begin{defn}\label{defn:contact_form}
  Let $(N,H)$ be a co-orientable contact manifold. Any 
  $\alpha \in \Omega^1(N)$ with $H = \ker \alpha$ is a 
  {\bf contact form} (for $(N,H)$). The pair $(N,\alpha)$ is 
  a {\bf co-oriented contact manifold}.
\end{defn}

\begin{Rmk}\label{rmk:contact_form}
  If $(N,\alpha)$ is a co-oriented contact manifold, then $\alpha : TN
  \to N \times \Rr$ is onto and, if $H = \ker \alpha$, then
  $(H,d\alpha) \to N$ is a symplectic vector bundle. Conversely, if
  $\alpha \in \Omega^1(N)$ has the above properties, then $(N,H = \ker
  \alpha)$ is a
  contact manifold. 
\end{Rmk}

A choice of contact form specifies
the following vector field. 

\begin{defn}\label{defn:reeb}
  Let $(N,\alpha)$ be a co-oriented contact manifold. The 
  {\bf Reeb vector field of $\boldsymbol{(N,\alpha)}$} is the unique 
  $R^{\alpha} \in \XX(N)$ such that 
  \begin{equation}\label{eq:Reeb_eqns}
  \alpha(R^{\alpha}) =1\quad \textnormal{ and }\quad d\alpha(R^{\alpha}, - ) = 0.
  \end{equation}
\end{defn}

A simple, but computationally useful property of co-oriented contact
manifolds is that they admit a natural complement to the contact
distribution. More precisely, if $(N,\alpha)$ is a co-oriented contact
manifold, setting $H = \ker \alpha$, then
\begin{equation}
  \label{eq:6}
  \begin{split}
    TN &\to \Rr\langle R^{\alpha} \rangle \oplus H \\
    v &\mapsto (\alpha(v)R^\alpha,v-\alpha(v)R^\alpha)
  \end{split}
\end{equation}
\noindent
and
\begin{equation}
  \label{eq:7}
  \begin{split}
    T^*N &\to \Rr\langle \alpha \rangle \oplus \mathrm{Ann}(\Rr\langle R^{\alpha} \rangle) \\
    \beta &\mapsto (\beta(R^\alpha)\alpha,\beta
    -\beta(R^\alpha)\alpha)
  \end{split}
\end{equation}
\noindent
are vector bundle isomorphisms.

Next we discuss diffeomorphisms that preserve
contact structures. 
\begin{defn}\label{defn:contactomorphisms}
  Let $(N_j,H_j)$ be a contact manifold for $j=1,2$. A {\bf
    contactomorphism between $\boldsymbol{(N_1,H_1)}$ and $\boldsymbol{(N_2,H_2)}$} 
  is a diffeomorphism $\Phi : N_1 \to N_2$ such that 
  $d\Phi(H_1) = H_2$. 
\end{defn}

\begin{Rmk}\label{rmk:line_bundle_iso}
  \mbox{}
  \begin{itemize}[leftmargin=*]
  \item Any contactomorphism $\Phi$ between $(N_1,H_1)$ and
    $(N_2,H_2)$ induces a vector bundle isomorphism
    $B: \Phi^*L_2 \stackrel{\cong}{\to} L_1$ covering the identity.
  \item Suppose that $(N_j,\alpha_j)$ is a 
    co-oriented contact manifold for $j=1,2$. Set $H_j = \ker \alpha_j$. Given 
    a contactomorphism $\Phi$ between $(N_1,H_1)$ and $(N_2,H_2)$, there exists a nowhere 
    vanishing $a \in C^{\infty}(N_1)$ such that $\Phi^*
    \alpha_2 = a \alpha_1$. Moreover, using $\alpha_j$ to identify $L_j$ with $N_j \times \Rr$ 
    and using the canonical isomorphism 
    $\Phi^*(N_2 \times \Rr) \cong N_1 \times \Rr$, the above
    isomorphism $B$ is 
    given by $ (x,\lambda) \mapsto (x,a^{-1}\lambda)$.
  \end{itemize}
\end{Rmk}

To conclude this section, we show that, up to taking double covers,
every connected contact manifold is co-orientable.

\begin{Rmk}\label{rmk:co-orientable_double_cover}
  Let $(N,H)$ be a connected contact manifold such that $L \to N$ is
  not trivializable. Since $L \to N$ is a line bundle, its structure
  group can be reduced to $\{ \pm 1\}$. Hence, there exists a
  double cover $q : \hat{N} \to N$  such that $\hat{L}:=q^*L$ is
  trivializable and $\hat{N}$ is connected. Since $q$ is a local
  diffeomorphism, $T\hat{N}$ is isomorphic to $q^*TN$. Setting $\hat{H}:= q^*H$, we have that $(\hat{N},\hat{H})$ is a
  contact manifold such that $T\hat{N}/\hat{H}$ is 
  isomorphic to $\hat{L}$ and, hence, trivializable.  We call $(\hat{N}, \hat{H})$ the 
  {\bf co-orientable double cover} of $(N,H)$.
\end{Rmk}

\subsection{Contact groupoids}\label{sec:contact-groupoids}

Throughout this paper, $G\rightrightarrows M$ and $G$ denote a 
{\bf Lie groupoid} over a manifold $M$, where $G$ and $M$ 
are the spaces of {\bf arrows} and {\bf objects} respectively. 
We refer to $M$ as the {\bf base} of $G$.
The structure maps of $\gpd$ are denoted as follows: 
$s,t: G\to M$ are the {\bf source} and {\bf target} maps 
respectively, $u : M \to G$ is the {\bf unit} map, 
$m : G^{(2)} \to G, \ m(g,h)=gh$ is the {\bf multiplication} 
map, where 
$G^{(2)}:=\{(g,h) \in G \times G \mid s(g) = t (h)\}$, and 
$i:G\to G,\ i(g)=g^{-1}$ is the {\bf inversion} map. Since 
$u : M \to G$ is an embedding, we often identify a point 
$x \in M$ with the unit $u(x) = 1_x$ and view $M$ as an embedded
submanifold of $G$. Given $x \in M$, the 
{\bf orbit of $\boldsymbol{x}$} is 
$S_x:= t(s^{-1}(x)) \subset M$. If $\Phi : G_1 \to G_2$ is a Lie
groupoid homomorphism, we say that it {\bf covers} the induced map 
$\phi:M_1 \to M_2$ on the bases. 

\begin{Rmk}[On Hausdorffness]\label{rmk:Hausdorff}
In general, the base and the source fibers of a 
Lie groupoid are assumed to be Hausdorff, while the space of 
arrows is not. In this paper, we assume that 
the space of arrows is also Hausdorff. This is 
primarily because we deal with proper maps and uniqueness of flows 
of vector fields.
\end{Rmk}

In this paper we are mostly interested in Lie groupoids 
that possess some degree of `compactness'.

\begin{defn}\label{defn:compactness}
  A Lie groupoid $\gpd$ is 
  \begin{itemize}[leftmargin=*]
      \item {\bf proper} if the map $(s,t) : G \to M \times M$ 
      is proper, i.e., the preimage of a compact set is compact, and
      \item {\bf compact} if $G$ is compact.
  \end{itemize}
\end{defn}

We are interested in Lie groupoids equipped with contact 
structures on the spaces of arrows that are, in some sense, 
`compatible' with the multiplication. To this end, we recall 
that, given a Lie groupoid $\gpd$, its 
{\em tangent Lie groupoid} $TG\rightrightarrows TM$ is the Lie 
groupoid with structure maps given by taking derivatives of 
those of $\gpd$.

\begin{defn}\label{defn:multiplicative_distn}
  A {\bf multiplicative distribution} on a Lie groupoid $\gpd$ 
  is a distribution  $H \subseteq TG$ that is a Lie 
  subgroupoid of $TG$ with base $TM$. 
\end{defn}

Multiplicative distributions satisfy several 
properties that are immediate consequences of 
Definition \ref{defn:multiplicative_distn}. If $H$ is a 
multiplicative distribution on $G$, then 
\begin{enumerate}[label=(\roman*),ref=(\roman*),leftmargin=*]
\item\label{item1} if $X_g\in H_g$ and $Y_h\in H_h$ are 
  composable, i.e., $ds(X_g)=dt(Y_h)$, then 
  $dm(X_g,Y_h)\in H_{gh}$, 
\item\label{item2} for all $x \in M$, $T_xM\subset H_x$, 
  and 
\item\label{item3} $H$ is both $s$-transversal and 
  $t$-transversal, i.e.,    
  $$TG=H+\ker ds \ \ \mathrm{and}\ \ TG=H+\ker dt.$$
\end{enumerate}

\begin{defn}\label{defn:contact_groupoid}
  A {\bf contact groupoid} is a pair $(G,H)$, where $G$ is a 
  Lie groupoid and $H$ is a multiplicative contact structure 
  on $G$. 
\end{defn}

\begin{Rmk}\label{rmk:ctct_gpd_literature}
Contact groupoids are also known as 
{\em conformal contact groupoids} (see
\cite[Appendix I, Definition $1.1$]{zamb_zhu}), or
{\em locally conformal contact groupoids} (see 
\cite[Example $7.6$]{crainic_zhu}). 
\end{Rmk}

In what follows we discuss the 1-form approach to contact groupoids. A {\bf representation} 
of a Lie groupoid $\gpd$ is a vector bundle $E \to M$ together with a linear 
action of $G$, i.e., a smooth assignment of a 
linear isomorphism $E_{s(g)}\to E_{t(g)},\ v\mapsto g\cdot v$ to each
arrow $g \in G$ that satisfies the usual axioms of 
an action.

\begin{defn}\label{defn:mult_form}
  Let $E$ be a representation of 
  a Lie groupoid $G$. A {\bf multiplicative form (with values in $\boldsymbol{E}$)} 
  is a form $\alpha\in\Omega^k(G;t^*E)$ such that 
  \begin{equation}\label{eq:mult_general}
  (m^*\alpha)_{(g,h)} = (\mathrm{pr}^*_1\alpha)_{(g,h)} + g\cdot(\mathrm{pr}_2^*\alpha)_{(g,h)}, 
  \end{equation}
  for all $(g,h)\in G^{(2)}$, where 
  $\mathrm{pr}_j :G^{(2)} \to G$ is the projection onto the 
  $j$th component for $j=1,2$.
\end{defn}

Given a multiplicative distribution $H$ on $G$, set 
$H^s:= H \cap \ker ds$. Property \ref{item3} implies that $TG/H$ 
is canonically isomorphic to $\ker ds/H^s$. Hence the vector 
bundle
$$E:=TG/H|_M \to M$$
is canonically isomorphic to 
$(\ker ds/H^s)|_M$. For any $(g,h) \in G^{(2)}$, right translation by $h$
\begin{equation}\label{eq:right}
\mathtt{r}_h:\ker d_gs\to \ker d_{gh}s ,\ \mathtt{r}_h(X)=dm(X,0_h) 
\end{equation}
\noindent
is an isomorphism that maps $H_g^s$ to $H_{gh}^s$. Hence, the maps 
\eqref{eq:right} induce a vector bundle isomorphism
\begin{equation}\label{eq:right_iso}
t^*E\overset{\mathtt{r}}{\cong} TG/H
\end{equation}
\noindent
covering the identity over $G$.

\begin{proposition}[Lemmas $3.6$ and $3.7$ in
  \cite{spencer_operators}]\label{prop:canonical_form} Let $H$ be a
  multiplicative distribution on $\gpd$. Then $E:=TG/H|_M$ inherits the structure of a representation of
  $G$. Moreover, the canonical projection $\alpha_{\mathrm{can}} : TG
  \to TG/H \cong t^*E$ is a multiplicative 1-form with values in $E$. Conversely, any multiplicative distribution is the kernel 
  of a pointwise surjective multiplicative 1-form. 
\end{proposition}

By Proposition \ref{prop:canonical_form}, if $(G,H)$ is a contact groupoid, then the line bundle $L:= TG/H|_M$
inherits the structure of a representation of $G$ and the generalized contact form $\alpha_{\mathrm{can}} : TG \to TG/H \cong
t^*L$ is a multiplicative 1-form with values in $L$. Conversely, any
multiplicative contact structure is the kernel of a multiplicative
generalized contact form. \\

To conclude this section, we introduce the multiplicative analog of
Definition \ref{defn:contactomorphisms}.

\begin{definition}\label{defn:iso_ctct_gpd}
  An {\bf isomorphism of contact groupoids} between $(G_1,H_1)$ and $(G_2,H_2)$ is a Lie groupoid
  isomorphism $\Phi : G_1 \to G_2$ that is a contactomorphism. We use
  the notation $\Phi : (G_1,H_1) \to (G_2,H_2)$.
\end{definition}

\subsection{Co-orientable contact groupoids}\label{sec:co-orient-cont}

In what follows we fix a co-orientable contact groupoid 
$(G,H)$ over $M$ unless otherwise stated. Then $L:= TG/H|_M$ is
trivializable. In fact, a choice of trivialization 
$\psi: L \to M \times \Rr$ induces a trivialization 
of $t^*L \cong TG/H$ (see equation \eqref{eq:right_iso}). Consequently,
by Proposition \ref{prop:canonical_form},
\begin{itemize}[leftmargin=*]
\item $M \times \Rr$ inherits the structure of representation of $G$, and
\item the contact form $\alpha  = \psi \circ \alpha_{\mathrm{can}}$ is multiplicative with
  respect to the above representation.
\end{itemize}
The above representation of $G$ on $M \times \Rr$ is given by 
fiberwise multiplication by a Lie groupoid homomorphism $ F: G \to
\Rr^*$, i.e., a nowhere vanishing function 
$F$ such that $F(gh) = F(g)F(h)$ for all $(g,h) \in G^{(2)}$. Multiplicativity of $\alpha$ becomes
\begin{equation}\label{eq:2} 
m^*\alpha = \mathrm{pr}^*_1\alpha + \mathrm{pr}^*_1(F)\mathrm{pr}_2^*\alpha,
\end{equation}
\noindent
(cf. equation \eqref{eq:mult_general}).

\begin{defn}\label{defn:mult_contact_form}
  Let $(G,H)$ be a co-orientable contact groupoid. A pair 
  $(\alpha, F)$ as above is a {\bf multiplicative contact form} (for $(G,H)$). 
  We also say that $\alpha$ is {\bf $\boldsymbol{F}$-multiplicative} 
  and that $(G,\alpha,F)$ is a {\bf co-oriented contact groupoid}. 
\end{defn}

The following result, stated without proof, shows the degrees of
freedom in choosing multiplicative contact forms for a given
co-orientable contact groupoid (see \cite[Appendix I, Lemma $1.5$, Part (ii)]{zamb_zhu}\footnote{In this paper we use the `opposite' convention to the one in 
\cite[Definition $2.1$]{zamb_zhu}, cf. 
Definition~\ref{defn:mult_form}.}).

\begin{lemma}\label{lemma:choice_mult_contact_form}
  Let $(\alpha,F)$ be a multiplicative contact form for  a co-orientable
  contact groupoid $(G,H)$ over $M$. Then $(\alpha',F')$ is a
  multiplicative contact form for $(G,H)$ if and only if
  there exists a nowhere vanishing 
  $a\in C^{\infty}(M)$ such that
  \begin{equation*}
    \alpha'= (t^*a) \alpha \qquad \text{ and } \qquad F'= \frac{t^*a}{s^*a} F. 
  \end{equation*}
\end{lemma}

Given a co-oriented contact groupoid $(G,\alpha,F)$, set 
\begin{equation}\label{eqn:1000}
r:=\ln(|F|)\ \ \mathrm{and}\ \ \sigma:=\mathrm{sgn}(F).
\end{equation}
Then $r(gh) = r(g) + r(h)$ for all $(g,h) \in G^{(2)}$,
i.e., $r$ is a $1$-cocycle in the 
{\em differentiable cohomology} of $G$ (see \cite{crainic} for 
a definition). Moreover, $\sigma:G\to \{\pm 1\}$ is a Lie groupoid 
homomorphism and $F = \sigma e^r$.

\begin{defn}\label{defn:reeb_cocycle}
  Let $(G,\alpha,F)$ be a co-oriented contact groupoid. We 
  call $r$ the 
  {\bf Reeb cocycle (of $\boldsymbol{(G,\alpha,F)})$} (see 
  \cite[Definition $1.3$]{crainic_zhu}), and $\sigma$ the 
  {\bf sign of $\boldsymbol{F}$}.
\end{defn}

Suppose that the Reeb cocycle $r$ of a co-oriented contact groupoid $(G,\alpha,F)$ is
a coboundary, i.e., there exists $\kappa\in C^\infty (M)$ such that
\begin{equation}\label{eq:reeb_trivial}
  r=s^*\kappa-t^*\kappa.
\end{equation} 
\noindent
Then, applying Lemma \ref{lemma:choice_mult_contact_form} with $a:=
e^{\kappa}$, we have that  $(e^{t^*\kappa}\alpha, \sigma)$ is a
multiplicative contact form for $(G,H = \ker \alpha)$. This proves the
following result.

\begin{proposition}\label{prop:proper_co-orientable}
  Let $(G,\alpha,F)$ be a co-oriented contact groupoid and let
  $\sigma=\mathrm{sgn}(F)$. If the Reeb cocycle $r$ of $(G,\alpha,F)$
  is a coboundary, then, for all $\kappa\in
  C^\infty (M)$ satisfying \eqref{eq:reeb_trivial}, $(e^{t^*\kappa}\alpha,\sigma)$ is a
  multiplicative contact form for $(G,H = \ker \alpha)$.  
\end{proposition}

Since the differentiable cohomology of proper Lie groupoids vanishes in
all positive degrees (see 
\cite[Proposition $1$]{crainic}), Proposition
\ref{prop:proper_co-orientable} immediately implies the
following result.

\begin{corollary}\label{cor:proper_co-orientable}
  Let $(G,H)$ be a proper co-orientable contact groupoid. Then there
  exist a Lie groupoid homomorphism $\sigma : G \to \{\pm 1\}$ and a
  contact form $\alpha$ for $(G,H)$ such that $(\alpha,\sigma)$ is a
  multiplicative contact form for $(G,H)$.
\end{corollary}

\begin{Rmk}\label{rmk:corank_one_mult_dist}
  Let $H$ be a multiplicative distribution on $\gpd$ such
  that $TG/H$ has rank one and is trivializable.  Then 
  \begin{itemize}[leftmargin=*]
  \item upon fixing a trivialization of $TG/H|_M$, $H$ is encoded by a
    multiplicative 1-form $(\alpha, F)$,
  \item the space of such multiplicative 1-forms is given by Lemma
    \ref{lemma:choice_mult_contact_form},
  \item equation \eqref{eqn:1000} gives a 1-cocycle $r$ in the
    differentiable cohomology of $G$ and a Lie groupoid homomorphism
    $\sigma : G \to \{ \pm 1\}$, and
  \item the analogs of Proposition
    \ref{prop:proper_co-orientable} and Corollary
    \ref{cor:proper_co-orientable} hold.
  \end{itemize}
  In particular, we say that $(\alpha,F)$ is a {\bf multiplicative
    1-form} (for $(G,H)$).
\end{Rmk}

Let $(G,\alpha,F)$ be a co-oriented contact groupoid. In some works in the
literature the sign of $F$  is assumed to
be positive (see \cite{crainic_zhu}). While this holds if the groupoid or its $s$-fibers are connected, in 
general it need not be, as shown below (see Section
\ref{sec:projectivisation_dual_Lie_algebra} for the notation).

\begin{example}\label{exm:mult_function&cpct_not_constant}
  Throughout this example, $n \geq 3$ is an odd natural number. The map $O(n)
  \to \{ \pm 1\} \times SO(n)$, $A \mapsto (\det A, (\det A)^{-1}A)$ is an
  isomorphism of Lie groups. Hence, upon identifying
  $\mathfrak{o}(n)^*$ with $\mathfrak{so}(n)^*$, the 
  coadjoint action of $O(n)$ factors through that of 
  $SO(n)$. The negative of the dual of the Killing form on 
  $\mathfrak{o}(n)^*$ is a Riemannian metric 
  that is bi-invariant with respect to both $O(n)$ and 
  $SO(n)$. We use this metric to identify the oriented projectivization
  $\mathbb{S}(T^*O(n))$ of $T^*O(n)$ with the unit sphere bundle
  $U(T^*O(n))$ of $T^*O(n)$. Upon using 
  right trivializations and the above isomorphism, there is a diffeomorphism
  $$U(T^*O(n)) \cong U(\mathfrak{o}(n)^*) \times (\{ \pm 1\} \times SO(n)),$$
  \noindent
  where $U(\mathfrak{o}(n)^*) \subset \mathfrak{o}(n)^*$ 
  denotes the unit sphere. We denote by 
  $\alpha \in \Omega^1(U(\mathfrak{o}(n)^*) \times (\{ \pm 1\} \times SO(n)))$ 
  the pullback of the restriction of the Liouville $1$-form 
  to $U(T^*O(n))$ along the above diffeomorphism. Then 
  $\ker \alpha$ is a contact structure. In what follows we define 
  a structure of Lie groupoid on 
  $U(\mathfrak{o}(n)^*) \times (\{ \pm 1\} \times SO(n))$ 
  over $U(\mathfrak{o}(n)^*)$ with the property that 
  $\alpha$ is $F$-multiplicative for some function $F$ that takes both positive and 
  negative values.

  Consider the right $\{ \pm 1\}\times SO(n)$-action 
  on $U(\mathfrak{o}(n)^*)$ given by
  $$\xi \cdot (\pm 1,A):= \pm \mathrm{Ad}^*_A(\xi).$$
  \noindent
  (This is {\em not} the coadjoint action of $O(n)$!) Endow 
  $G: =U(\mathfrak{o}(n)^*) \times (\{ \pm 1\} \times SO(n))$ 
  with the structure of an action Lie groupoid and define a Lie groupoid homomorphism
  $F : G \to \Rr^*$ by $F(\xi;\pm 1,A)= \pm 1$. Then a direct 
  calculation (very similar to that in 
  \cite[Example $2.3$]{zamb_zhu}), shows that $\alpha$ is 
  $F$-multiplicative. Hence, by 
  Proposition \ref{prop:canonical_form}, $(G,\alpha, F)$ is a
  co-oriented contact groupoid. By construction, $G$ is 
  compact, the base is connected and $F$ takes both positive and
  negative values. 
\end{example}

To conclude this section, we discuss properties of the Reeb vector
field of a multiplicative contact form. Recall that if $G$ is a Lie groupoid and $X \in \XX(G)$, then the 
vector field $X^{\mathtt{L}} \in  \XX(G)$ given by
\begin{equation}
  \label{eq:8}
  X^{\mathtt{L}}_g := dm(0_g,X_{1_{s(g)}} - d
  t (X_{1_{s(g)}})) 
\end{equation}
\noindent
is {\em left-invariant}, i.e., $X^{\mathtt{L}} \in \Gamma(\ker dt)$,
and $\mathtt{l}_g(X^{\mathtt{L}}_h) =
X^{\mathtt{L}}_{gh}$ for all $(g,h) \in G^{(2)}$, where $\mathtt{l}_g$ is left translation by $g$
(cf. equation \eqref{eq:right}). 

\begin{lemma}\label{lemma:reeb_gpd}
  Let $(G,\alpha,F)$ be a co-oriented contact groupoid over $M$.
  \begin{itemize}[leftmargin=*]
  \item The Reeb vector field $R^{\alpha}$ of
    $(G,\alpha)$ is right-invariant, i.e., $R^{\alpha}\in \Gamma(\ker ds)$, and, for all $(g,h)\in G^{(2)}$, 
    \begin{equation}\label{eq:right_invariant_Reeb} 
      \mathtt{r}_h(R^{\alpha}_g)=R^{\alpha}_{gh}.
    \end{equation}
  \item The left-invariant vector field $R^{\alpha,\mathtt{L}}$
    induced by $R^{\alpha}$ is given by
    \begin{equation}
      \label{eq:16}
       R^{\alpha,\mathtt{L}} = FR^{\alpha} + \Lambda^{\alpha,\sharp}(dF), 
    \end{equation}
    \noindent
    where $\Lambda^{\alpha} \in \XX^2(G)$ is as in Example \ref{exm:cooriented_contact}. 
  \end{itemize}
\end{lemma}

\begin{proof}
  The first statement follows immediately from \cite[Corollary
  $5.2$]{crainic_salazar}. To prove the second statement, by equation \eqref{eq:6},
  there exists $Y \in \Gamma(\ker \alpha)$ such that
  $$R^{\alpha,\mathtt{L}} = \alpha(R^{\alpha,\mathtt{L}})R^{\alpha} +
  Y.$$
  \noindent
  We recall that $M$ is a {\em Legendrian} submanifold of $(G,\ker \alpha)$, i.e.,
  $TM$ is contained in $\ker \alpha$ and is equal to its
  symplectic orthogonal with respect to $d\alpha|_{\ker \alpha}$ (see
  Remark \ref{rmk:contact_form} and \cite[Proposition
  5.1]{crainic_salazar}). This property, the definition of
  $R^{\alpha,\mathtt{L}}$, $F$-multiplicativity of $\alpha$ -- equation
  \eqref{eq:2} --, and the identity obtained from equation \eqref{eq:2}
  by taking exterior derivatives, yield that, for any $X \in
  TG$, 
  \begin{equation}\label{eq:final2}
  \alpha(R^{\alpha,\mathtt{L}})=F, \ \ d\alpha(R^{\alpha,\mathtt{L}}, X)=-dF(X).
  \end{equation}
  \noindent
  Since
  $d\alpha(R^{\alpha}, - ) = 0$, the definition of
  $\Lambda^{\alpha}$ (see Example \ref{exm:cooriented_contact}), and equation \eqref{eq:final2}, imply the
  desired result.
\end{proof}

\subsection{Co-orientable finite cover of a contact groupoid}\label{subsec:co-or_finite_cover}

In this section we establish a multiplicative analog of the construction 
in Remark \ref{rmk:co-orientable_double_cover} that we use in 
the proof of our main result (see Theorem \ref{thm:gray_dist}). To the
best of our knowledge, this has not appeared elsewhere in the
literature.

Let $(G,H)$ be a contact groupoid over a
connected manifold
$M$ such that $L \to M$ is not trivializable. As in  
Remark \ref{rmk:co-orientable_double_cover}, there exists a 
double cover $q: \hat{M} \to M$ such that 
$\hat{L}:=q^*L \to \hat{M}$ is trivializable and $\hat{M}$ is connected. Since $q$ is a 
local diffeomorphism, the hypotheses of 
\cite[Proposition $2.3.1$]{mack} are satisfied and we can 
consider the {\bf pullback Lie groupoid} 
$\hat{G}:= q^!G \rightrightarrows \hat{M}$, where 
$$ \hat{G} = \{(x,g,y) \in \hat{M} \times G \times \hat{M} \mid s(g) = q(y)\, , \, t(g) = q(x)\},$$
\noindent
and the structure maps are:
\begin{equation}
  \label{eq:24}
    \begin{split}
        \hat{s}(x,g,y) := y \, , \, \hat{t}(x,g,y) := x \, &, \,  \hat{u}(z):= (z, 1_{q(z)}, z), \\
        \hat{m}((x,g,y),(y,h,z)):= (x,gh,z) \, &, \, \hat{i}(x,g,y):= (y,g^{-1},x).
    \end{split}
\end{equation}
\noindent
We collect a few properties of $\hat{G}$ below.

\begin{enumerate}[label=(\arabic*),ref=(\arabic*),leftmargin=*]
\item \label{item:fibration}
  The restriction of the projection 
  $\hat{M} \times G \times   \hat{M} \to G$ to $\hat{G}$, denoted by
  $Q$, is a Lie groupoid homomorphism onto $G$ that covers $q$.
\item \label{item:z_2_squared} The free and proper 
  $\Zz_2$-action on $\hat{M}$ with quotient map 
  $q: \hat{M} \to M$ lifts to a free and proper 
  $\Zz_2 \times \Zz_2$-action on $\hat{G}$ with quotient map 
  $Q: \hat{G} \to G$. The lifted action is the restriction of the
  free and proper $\Zz_2 \times \Zz_2$-action on 
  $\hat{M} \times G \times \hat{M}$ defined as follows: the first 
  (respectively second) $\Zz_2$ acts on the first 
  (respectively second) copy of $\hat{M}$ and trivially on 
  $G$. 
\item \label{item:pullback_tangent} The tangent Lie groupoid 
  $T\hat{G} \rightrightarrows T\hat{M}$ equals the pullback 
  Lie groupoid $(dq)^!TG \rightrightarrows T\hat{M}$.
\end{enumerate}

\begin{Rmk}\label{rmk:action_not_gpd}
The $\Zz_2 \times \Zz_2$-action on $\hat{G}$ is {\em not} by 
Lie groupoid isomorphisms, as it does not preserve units. 
However, the diagonal copy of $\Zz_2$ in $\Zz_2 \times \Zz_2$ 
acts on $\hat{G}$ by Lie groupoid isomorphisms.
\end{Rmk}

Set $\hat{H}:= (dq)^!H \rightrightarrows T\hat{M}$. By 
property \ref{item:pullback_tangent}, $\hat{H}$ is a 
multiplicative distribution on $\hat{G}$. Moreover, by 
property \ref{item:z_2_squared}, $Q$ is a local diffeomorphism. 
Hence, $dQ$ induces an isomorphism $T\hat{G} \cong Q^*TG$ that 
restricts to an isomorphism $\hat{H} \cong Q^*H$ (all as 
vector bundles over $G$). Thus, $\hat{H}$ is a contact 
structure on $\hat{G}$. Therefore, $(\hat{G},\hat{H})$ is 
a contact groupoid. If $\hat{L}:= T\hat{G}/\hat{H}|_{\hat{M}}$, 
then $\hat{L}$ is canonically isomorphic to $q^*L$. Hence, $(\hat{G},\hat{H})$ is a 
co-orientable contact groupoid.

\begin{defn}\label{defn:co-orientable_cover_gpd}
  Let $(G,H)$ be a contact groupoid over
  a connected manifold $M$ such that $L \to M$ is not trivializable. The 
  co-orientable contact groupoid $(\hat{G},\hat{H})$ 
  constructed above is called the 
  {\bf co-orientable finite cover of $\boldsymbol{(G,H)}$}.
\end{defn}

\begin{Rmk}\label{lemma:compactness_preserved}
  Let $(G,H)$ be a
  contact groupoid over a connected manifold $M$ such that $L \to M$
  is not trivializable. Let
  $(\hat{G},\hat{H})$ be the co-orientable finite cover of
  $(G,H)$. Then $G$ is compact if and only if $\hat{G}$ is compact.
\end{Rmk}

\begin{Rmk}\label{rmk:co-orientable_cover_corank_one}
  Let $H$ be a corank one multiplicative distribution on $\gpd$ such
  that $TG/H|_M$ is not trivializable. The above construction holds in this more general setting with the obvious
  adaptations (cf. Remark \ref{rmk:corank_one_mult_dist}). In
  particular, it makes sense to consider the co-orientable finite
  cover of the Lie groupoid $G$ with the distribution $H$ that, by
  abuse of notation, we also denote by $(\hat{G},\hat{H})$.
\end{Rmk}

To conclude this section. we prove Lemma \ref{lemma:the_fs_are_constant}, which is a technical result relating the $\Zz_2 \times
\Zz_2$-action and some multiplicative contact forms on the space of
arrows of the co-orientable finite cover $(\hat{G},\hat{H})$ of
$(G,H)$. (We use this result in the proof of Theorem \ref{thm:main_not_triv}.) 
Let $\gamma\in\mathrm{Diff}(\hat{M})$ be the non-trivial 
diffeomorphism given by the $\Zz_2$-action on $\hat{M}$, and let 
$\varphi:\Zz_2\times\Zz_2\to\mathrm{Diff}(\hat{G})$ be the $\Zz_2
\times \Zz_2$-action on $\hat{G}$. For any
$(a,b) \in \Zz_2 \times \Zz_2$,  
$\varphi_{ab}:=\varphi(a,b)$ is the restriction of 
\begin{equation}
  \label{eq:23}
  (\gamma^a,\mathrm{id},\gamma^b) \in \mathrm{Diff}(\hat{M} \times G
\times \hat{M})
\end{equation}
\noindent
to $\hat{G}$. While $\varphi_{ab}$ is not a groupoid homomorphism if 
$a\neq b$ (see Remark \ref{rmk:action_not_gpd}), the multiplication
and the $\Zz_2 \times \Zz_2$-action on $\hat{G}$ are related as
follows. First, observe that the
map $(\varphi_{ab},\varphi_{bc})$ restricts to a diffeomorphism of
$\hat{G}^{(2)}$ that satisfies $ \varphi_{ac} \circ
\hat{m} = \hat{m} \circ (\varphi_{ab},\varphi_{bc})$ for all $a,b,c \in \Zz_2$.
Taking derivatives, we have that
\begin{equation}
  \label{rels_with_mult}
  d\varphi_{ac}(d\hat{m}(X,Y))=d\hat{m}(d\varphi_{ab}(X),d\varphi_{bc}(Y)),
\end{equation}
\noindent
for all 
$(X,Y)\in T\hat{G}^{(2)}$ and for all 
$a,b,c\in\Zz_2$.
By property \ref{item:z_2_squared} above, $\varphi_{ab}$ is a contactomorphism of $(\hat{G},\hat{H})$. If $\hat{\alpha}\in\Omega^1(\hat{G})$
is any contact form for $(\hat{G},\hat{H})$, for each $(a,b)
\in \Zz_2 \times \Zz_2$, there exists a nowhere vanishing $f_{ab}\in C^\infty(\hat{G})$ satisfying
\begin{equation}
  \label{The_fs}
  \varphi_{ab}^*\hat{\alpha}=f_{ab}\hat{\alpha}.
\end{equation} 

\begin{lemma}\label{lemma:the_fs_are_constant}
  Let $(G,H)$ be a contact groupoid over a connected manifold
  $M$ such that $L \to M$ is not
  trivializable and let $(\hat{G},\hat{H})$ be its co-orientable
  finite cover. Let $\hat{\sigma}:\hat G\to \{\pm 1\}$ be a Lie groupoid homomorphism
  and suppose that $(\hat{\alpha}, \hat{\sigma})$ is a multiplicative contact
  form for $(\hat{G},\hat{H})$. Then, for any $(a,b) \in \Zz_2
  \times \Zz_2$, the function
  $f_{ab}$ given by equation \eqref{The_fs} is
  constant. In fact, $f_{00} = f_{01} \equiv 1$ and $f_{10} =
  f_{11} \equiv -1$.
\end{lemma}

\begin{proof}
  Set $\hat{R}: = R^{\hat{\alpha}}$ and let
  $(x,g,y) \in \hat{G}$. Using equations \eqref{eq:right_invariant_Reeb} 
  and \eqref{rels_with_mult}, we have that, for any $a,b,c \in
  \Zz_2$,  
  \begin{equation}
    \label{eq:4}
    \begin{split}
       f_{ac}(x,g,y) & =
      (\varphi^*_{ac}\hat{\alpha})_{(x,g,y)}(\hat{R}_{(x,g,y)}) =
      \hat{\alpha}_{\varphi_{ac}(x,g,y)}(d\varphi_{ac}(\hat{R}_{(x,g,y)})) \\
      & =\hat{\alpha}_{\varphi_{ac}(x,g,y)}
      (d\varphi_{ac}(d\hat{m}(\hat{R}_{(x,g,y)}, 0_{(y,1_{q(y)},y)}))) \\
      & =\hat{\alpha}_{\varphi_{ac}(x,g,y)}(d\hat{m}(d\varphi_{ab}(\hat{R}_{(x,g,y)}),d\varphi_{bc}(0_{(y,1_{q(y)},y)}))) \\
      & =\hat{\alpha}_{\varphi_{ab}(x,g,y)}(d\varphi_{ab}(\hat{R}_{(x,g,y)})) = (\varphi_{ab}^*\hat{\alpha})_{(x,g,y)}(\hat{R}_{(x,g,y)}) \\
      & = f_{ab}(x,g,y),
    \end{split}
  \end{equation}
  \noindent
  where in the fourth equality we use that $\hat{\alpha}$ is $\hat{\sigma}$-multiplicative. Hence, for any $a,b,c \in
  \Zz_2$, $f_{ab} \equiv f_{ac}$. Since $f_{00} \equiv 1$, we have
  that $f_{01} \equiv 1$. 
  
  Since $\hat{\alpha}$ is
  $\hat{\sigma}$-multiplicative and $d\hat{\sigma} \equiv 0$, 
  Lemma \ref{lemma:reeb_gpd} implies that 
  \begin{equation*}
    \hat{R}_{(x,g,y)} = d\hat{m}(0_{(x,g,y)},
     \hat{\sigma}(x,g,y)\hat{R}_{(y,1_{q(y)},y)}). 
  \end{equation*}
  \noindent
  Thus arguing as in \eqref{eq:4}, we have that
  $$ f_{ac}(x,g,y) =
  \hat{\sigma}(\gamma^a(x),g,\gamma^b(y))\hat{\sigma}(x,g,y)f_{bc}(y,1_{q(y)},y), $$
  \noindent
  for all
  $a,b,c \in \Zz_2$.  Taking $a=c = 1$ and $b = 0$ and using the
  fact that $f_{01} \equiv 1$, we have that $f_{11}(x,g,y) =
  \hat{\sigma}(\gamma(x),g,y)\hat{\sigma}(x,g,y)$. We observe that,
  since $\hat{\sigma} : \hat{G} \to \{ \pm 1\}$ is a groupoid homomorphism,
  $$\hat{\sigma}(\gamma(x),g,y)\hat{\sigma}(x,g,y) =
  \hat{\sigma}(\gamma(x),1_{q(x)},x) (\hat{\sigma}(x,g,y))^2 = \hat{\sigma}(\gamma(x),1_{q(x)},x).$$
  \noindent
  Hence, $f_{11}(x,g,y) =\hat{\sigma}(\gamma(x),1_{q(x)},x)$. This shows
  that $f_{11} = f_{10}$ is, at the same time, locally constant and
  the pullback of a function on $\hat{M}$ along $\hat{t}$. Hence, since $\hat{M}$ is connected, $f_{11} =
  f_{10}$ is constant; moreover, it has to be equal to $1$ or to
  $-1$. Suppose that $f_{11} = f_{10} \equiv 1$. Then the $\Zz_2
  \times \Zz_2$-action on $\hat{G}$ preserves the contact form
  $\hat{\alpha}$ and, therefore, it induces a contact form $\alpha$
  for $(G,H)$. This implies that $L= TG/H|_M \to M$ is trivializable,
  which is a contradiction.
\end{proof}

\section{Deformations of contact groupoids}\label{sec:deform-cont-group}
The aim of this section is to define deformations of (co-oriented) contact
groupoids and to study some of their properties, with
emphasis on deformations of compact contact groupoids. The idea is to
formalize the notion of `smooth 1-parameter family of (compact) contact
groupoids'. The
key result of this section is Corollary \ref{cor:existence_form},
which gives the existence of a `good' smooth 1-parameter family of
multiplicative contact forms for deformations of co-orientable compact
contact groupoids. We use this result in the proof of Theorem
\ref{thm:gray_dist}. The main references for deformations of Lie groupoids are \cite{DefLieGpds,Riemstacks}.

\subsection{Deformations of Lie groupoids}\label{sec:deform-lie-group}

\begin{defn}[Definition 1.9 in \cite{DefLieGpds}]\label{defn:family_Lie_gpd}
  A {\bf family of Lie groupoids (over
      $\boldsymbol{B}$)} is a Lie groupoid
    $\tilde{G} \rightrightarrows \tilde{M}$ together with a surjective
    submersion $p : \tilde{M} \to B$ such that
    $p \circ \tilde{s} = p \circ \tilde{t}$. We use the notation $\tilde{G} \rightrightarrows \tilde{M} \to B$. 
\end{defn}

\begin{remark}\label{rmk:fibres_family_gpd}
  Given $\tilde{G} \rightrightarrows \tilde{M} \to B$, for any
  $b \in B$, set $G_b:= (p\circ \tilde{s})^{-1}(b)$ and
  $M_b:= p^{-1}(b)$. There is a unique structure of Lie groupoid on
  $G_b \rightrightarrows M_b$ making it into a Lie subgroupoid of
  $\tilde{G} \rightrightarrows \tilde{M}$ that we call the {\bf
    fiber of $\boldsymbol{\tilde{G} \rightrightarrows \tilde{M} \to
      B}$ over $\boldsymbol{b}$}. 
\end{remark}



We formalize the notion of `a smooth 1-parameter family of Lie
groupoids' as follows.

\begin{defn}\label{defn:deformations_Lie_gpd}
  Let $\gpd$ be a Lie groupoid and let $I \subseteq \Rr$ be an open
  interval containing $0$. A {\bf deformation of $\boldsymbol{G}$}
  is a family of Lie groupoids $\tilde{G} \rightrightarrows \tilde{M}
  \to I$ such that
  \begin{itemize}[leftmargin=*]
  \item as manifolds, $\tilde{G} = G \times I$ and $\tilde{M} = M
    \times I$,
  \item the submersions $p \circ \tilde{s}: G \times I \to I$ and $p:
    M \times I \to I$ are projections onto the
    second component, and
  \item as Lie groupoids, $G_0 = G$.
  \end{itemize}
  We use the notation $\tilde{G} = \{G_\tau\}$, where $\tau \in I$. A
  deformation $\tilde{G}$ is {\bf proper} if 
  $\tilde{G}$ is proper. The {\bf constant} deformation of $G$ is the one in which
  $G_\tau = G$ as Lie groupoids for all $\tau$.
\end{defn}

\begin{rmk}\label{rmk:contrast_defn}
  \mbox{}
  \begin{itemize}[leftmargin=*]
  \item Definition \ref{defn:deformations_Lie_gpd} is
    a {\em strict}
    deformation in the sense of \cite[Definition 1.6]{DefLieGpds}. On the other
    hand, it is a special case of \cite[Definition
    5.1.1]{Riemstacks}, which allows for locally trivial submersions over
    any manifold.
  \item  The Lie groupoid structure on a deformation $\tilde{G}$ of
    $G$ is determined by the Lie groupoid structures on $G_\tau$ for all
    $\tau$.
  \end{itemize}
\end{rmk}

\begin{defn}\label{defn:equivalence_deformations}
  Let $G$ be a Lie groupoid.
  \begin{itemize}[leftmargin=*]
  \item Two deformations $\tilde{G}_1 = \{G_{1,\tau}\}$, $\tilde{G}_2
    = \{G_{2,\tau}\}$ of $G$ are {\bf isomorphic} if there exists a Lie groupoid
    isomorphism $\tilde{\Phi} : \tilde{G}_1 \to \tilde{G}_2$ of the 
    form 
    \begin{equation}
      \label{eq:5}
      \tilde{\Phi}(g,\tau) = (\Phi_{\tau}(g), \tau)
    \end{equation}
    \noindent
    such that
    $\Phi_0: G_{1,0} = G \to G_{2,0} = G$ is
    the identity. We use the notation $\tilde{\Phi} = \{\Phi_\tau\}$.
  \item A deformation $\tilde{G}$ of
    $G$ is {\bf trivial} if it is isomorphic to the constant deformation. 
  \end{itemize}
\end{defn}

\begin{Rmk}\label{rmk:iso_determined_time}
  By equation \eqref{eq:5}, an isomorphism $\tilde{\Phi}$ of deformations of $G$ is completely
  determined by the Lie groupoid isomorphisms $\Phi_\tau$ for all
  $\tau$ (cf. Remark \ref{rmk:contrast_defn}). This justifies the above notation.
\end{Rmk}

\begin{Rmk}\label{rmk:compare_marius}
  Definition \ref{defn:equivalence_deformations} should be compared
  with \cite[Definitions 1.6 and 1.9]{DefLieGpds}. In {\em loc. cit.} the
  authors allow for isomorphisms that are not the identity on the
  fiber over zero (see \cite[Definition 1.9]{DefLieGpds}), and
  consider a more general
  equivalence relation on deformations called {\em equivalence} (see
  \cite[Definition 1.6]{DefLieGpds}).
\end{Rmk}

The following stability result is proved in \cite{DefLieGpds} using
vanishing results for deformation cohomology
of Lie groupoids (see also \cite{Riemstacks} for similar stability
results proved using Riemannian metrics on Lie groupoids).

\begin{thm}[Theorem 1.7 in \cite{DefLieGpds}]\label{thm:stability_compact}
  Any deformation of a compact Lie groupoid is trivial.
\end{thm}

\begin{Rmk}\label{rmk:proof_stability}
  Strictly speaking, what is proved in \cite{DefLieGpds} is that
  deformations of compact Lie groupoids are trivial up to equivalence (see Remark
  \ref{rmk:compare_marius}). However, since the desired isomorphism is
  constructed using the flow of a time-dependent vector field and, in
  this case, the Lie groupoid is compact, Theorem
  \ref{thm:stability_compact} follows immediately from \cite[Remark 5.5
  and Theorem 7.1]{DefLieGpds}.
\end{Rmk}

\subsection{Deformations of contact groupoids}\label{sec:strict-deform-cont}
If $\tilde{G} \rightrightarrows \tilde{M} \to
B$ is a family of Lie groupoids over $B$, then $T\tilde{G}
\rightrightarrows T\tilde{M} \to TB$ is a family of Lie groupoids over
$TB$ with structure maps obtained by taking derivatives
of those of $\tilde{G} \rightrightarrows \tilde{M} \to
B$. Moreover, for any $b \in B$, the fiber of $T\tilde{G}
\rightrightarrows T\tilde{M} \to TB$ over  $0_b$ is the tangent Lie groupoid of
$G_b \rightrightarrows M_b$.

\begin{defn}\label{defn:family_mult_distns}
  Let $\tilde{G} \rightrightarrows \tilde{M} \to
  B$ be a family of Lie groupoids. A {\bf family of multiplicative
    distributions on $\boldsymbol{\tilde{G} \rightrightarrows \tilde{M} \to
      B}$} is a multiplicative distribution $\tilde{H} \subseteq
  T\tilde{G}$ on $\tilde{G}$.
\end{defn}

The following result motivates the terminology of Definition
\ref{defn:family_mult_distns}.

\begin{lemma}\label{lemma:family_mult_distns}
  Let $\tilde{H}$ be a family of multiplicative distributions on $\tilde{G} \rightrightarrows \tilde{M} \to
  B$. Then $\tilde{H} \rightrightarrows T\tilde{M} \to TB$ is a family of
  Lie groupoids over $TB$. Moreover, for any $b \in B$, if $H_b$ denotes the fiber of $\tilde{H}
  \rightrightarrows T\tilde{M} \to TB$ over $0_b$, then
  \begin{enumerate}[leftmargin=*]
  \item \label{item:3} $H_b\subseteq
    T G_b$ is a multiplicative distribution on $G_b$, and
  \item \label{item:4} the vector bundle $E_b:=TG_b/H_b|_{M_b}$ is canonically
    isomorphic to the restriction of $\tilde{E}:=
    T\tilde{G}/\tilde{H}|_{\tilde{M}}$ to $M_b$.
  \end{enumerate}
\end{lemma}

\begin{proof}
  Since $\tilde{H}$ is a Lie subgroupoid of $T\tilde{G}$ over
  $T\tilde{M}$ and since $T\tilde{G}
  \rightrightarrows T\tilde{M} \to TB$ is a family of Lie groupoids,
  it follows that $\tilde{H} \rightrightarrows T\tilde{M} \to TB$ also
  is, and that \eqref{item:3} holds for any $b \in B$. It remains to
  prove \eqref{item:4}. For any $b \in B$,
  the surjective submersion $p \circ \tilde{s} : \tilde{G} \to B$
  induces the following diagram
  $$ \xymatrix{ 0 \ar[r] & H_b \ar@{^{(}->}[r] \ar@{^{(}->}[d] &
    \tilde{H}|_{G_b}  \ar@{>>}[r] \ar@{^{(}->}[d] & G_b \times T_b B
    \ar@{=}[d] \ar[r] & 0 \\
    0 \ar[r] & TG_b \ar@{^{(}->}[r] &
    T\tilde{G}|_{G_b}  \ar@{>>}[r]^-{d(p \circ \tilde{s} \circ i_b)} & G_b \times T_b B
    \ar[r] & 0.}$$
  \noindent
  Hence, $TG_b/H_b$ is canonically isomorphic to
  $T\tilde{G}/\tilde{H}|_{G_b}$. The result follows by using the fact that
  $G_b$ is a Lie subgroupoid of $\tilde{G}$.
\end{proof}

The following
result, stated using the notation of Lemma
\ref{lemma:family_mult_distns}, is an immediate consequence of Lemma \ref{lemma:family_mult_distns}
and of Proposition \ref{prop:canonical_form}.

\begin{cor}\label{cor:family_forms}
  Let $\tilde{H}$ be a family of multiplicative distributions on $\tilde{G} \rightrightarrows \tilde{M} \to
  B$. For any $b \in B$, upon identifying $E_b$ and $
  \tilde{E}|_{M_b}$ using the canonical isomorphism of Lemma \ref{lemma:family_mult_distns},
  \begin{itemize}[leftmargin=*]
  \item the representation of $G_b$ on $E_b$ equals the restriction of
    the representation of $\tilde{G}$ on $\tilde{E}$ to $G_b$, and
  \item the multiplicative 1-forms $i_b^*\tilde{\alpha}_{\mathrm{can}}$ and $
    \alpha_{b,\mathrm{can}}$ with values in
    the representation $E_b$ are equal.
  \end{itemize}
\end{cor}

\begin{definition}\label{defn:def_contact_gpd}
  A {\bf deformation of a contact groupoid $\boldsymbol{(G,H)}$} is a
  deformation $\tilde{G} =\{G_\tau\}$ of $G$ together with a family
  of multiplicative distributions $\tilde{H}$ on $\tilde{G}$, such
  that
  \begin{itemize}[leftmargin=*]
  \item for all $\tau$, $(G_\tau,H_{\tau})$
    is a contact groupoid,
  \item for all $(g,\tau) \in G \times I$,
    \begin{equation}
      \label{eq:1}
      \tilde{H}_{(g,\tau)} = (H_{\tau})_g \oplus
      T_\tau I \subseteq T_g G \oplus T_\tau I, \text{ and} 
    \end{equation}
  \item as contact groupoids, $(G_0,H_0) = (G,H)$.
  \end{itemize}
  We use the notation
  $(\tilde{G},\tilde{H}) = \{(G_\tau, H_{\tau})\}$; if $\{G_\tau\}$ is the constant deformation, we write
  $(G \times I, \tilde{H}) = \{(G,H_\tau)\}$. The {\bf constant} deformation of $(G,H)$ is the one in which
  $(G_\tau,H_{\tau}) = (G,H)$ as contact groupoids for all $\tau$.
\end{definition}

\begin{rmk}\label{rmk:strict_mult_distn}
  The second condition in Definition
  \ref{defn:def_contact_gpd} can be reformulated as
  follows: $\{0_g\} \oplus T_{\tau}I$ is contained in
  $\tilde{H}_{(g,\tau)}$ for all $(g,\tau) \in G \times I$. Hence,
  $\tilde{H}$ is completely determined by $H_\tau$ for all $\tau$,
  thus justifying the notation (cf. Remark
  \ref{rmk:contrast_defn}). Many of the results below hold without imposing equation
  \eqref{eq:1}. However, this more general setting goes beyond the
  scope of this paper. 
\end{rmk}

Let $(\tilde{G},\tilde{H})$ be a deformation of
a contact groupoid $(G,H)$. Since $\tilde{M} = M
\times I$, setting $\tilde{L} := T\tilde{G}/\tilde{H}|_{\tilde{M}}$ and $L :=
TG/H|_M$, there exists an isomorphism of vector bundles
\begin{equation}\label{eq:9}
  \psi : \tilde{L} \to \mathrm{pr}^*L,
\end{equation}
\noindent
where $\mathrm{pr} : M \times I \to M$ is projection onto the first
component (see, e.g., \cite[Chapter 4, Section 1, Theorem 1.5]{hirsch}). 
Combining this observation with Lemma \ref{lemma:family_mult_distns}, 
we have proved the following result.

\begin{lemma}\label{lemma:deformation_underlying_strict}
  If $(\tilde{G},\tilde{H}) = \{(G_\tau,H_\tau)\}$ is a deformation of
  a contact groupoid $(G,H)$, then
  $L_{\tau} = TG_{\tau}/H_{\tau}|_{M_{\tau}}$ is isomorphic to $L =
  TG/H|_M$ for
  all $\tau$.
\end{lemma}

\begin{definition}\label{defn:iso_contact}
  Let $(G,H)$ be a contact groupoid.
  \begin{itemize}[leftmargin=*]
  \item Two deformations $(\tilde{G},\tilde{H})$ and
    $(\tilde{G}',\tilde{H}') $ of $(G,H)$ are {\bf isomorphic} if there
    exists an isomorphism $\tilde{\Phi}: \tilde{G} \to \tilde{G}'$ of
    deformations of $G$ such that $d\tilde{\Phi}(\tilde{H}) = \tilde{H}'$.
  \item A deformation $(\tilde{G},\tilde{H})$ of $(G,H)$ is {\bf
      trivial} if it is isomorphic to the constant one.
  \end{itemize}
\end{definition}

\begin{Rmk}\label{rmk:iso_deformations}
  An isomorphism $\tilde{\Phi}$ between two deformations $(\tilde{G},\tilde{H})$ and
    $(\tilde{G}',\tilde{H}') $ of a contact groupoid $(G,H)$ is
    completely determined by the isomorphisms of contact groupoids
    $\Phi_\tau : (G_\tau,H_\tau) \to (G'_\tau,H'_\tau)$ for all $\tau$
    (cf. Remark \ref{rmk:iso_determined_time}).
\end{Rmk}

The following (partial) stability results for deformations of
compact contact groupoids is a consequence of Theorem \ref{thm:stability_compact}.

\begin{cor}\label{cor:strict_deformations_mult_structures}
  Any deformation $(\tilde{G},\tilde{H}) =
  \{(G_\tau,H_\tau)\}$ of a compact contact groupoid
  $(G,H)$ is isomorphic to one of the form
  $(G \times I, \tilde{H}') =\{(G,H'_{\tau})\}$.
\end{cor}

\begin{proof}
  By Theorem \ref{thm:stability_compact}, there exists
  an isomorphism $\tilde{\Phi} = \{\Phi_\tau\}$ between $\tilde{G}$ and $G \times I$ that satisfies equation
  \eqref{eq:5}. Hence, $\tilde{H}':=
  d\tilde{\Phi}(\tilde{H})$ satisfies equation
  \eqref{eq:1} and $(G,H'_\tau)$ is a contact
  groupoid for all $\tau$. This completes the proof.
\end{proof}

\subsection{Deformations of co-oriented contact groupoids}\label{sec:deform-co-orient}

\begin{defn}\label{defn:deformation_co-oriented}
  A {\bf deformation of a co-oriented contact groupoid $\boldsymbol{(G,\alpha,F)}$} is a deformation
  $\tilde{G} = \{G_{\tau}\}$ of $G$ together with a Lie groupoid
  homomorphism $\tilde{F}: \tilde{G} \to \Rr^*$ and a surjective $\tilde{F}$-multiplicative 1-form $\tilde{\alpha} \in
  \Omega^{1}(\tilde{G})$ such that
  \begin{itemize}[leftmargin=*]
  \item for all $\tau$, $(G_{\tau},\alpha_\tau,F_{\tau}):=
    (G_{\tau}, i^*_{\tau}\tilde{\alpha},i^*_{\tau}\tilde{F})$
    is a co-oriented contact groupoid,
  \item for all $(g,\tau) \in G \times I$,
    \begin{equation}
      \label{eq:25}
      \{0_g\} \oplus
      T_\tau I \subset \ker
    \tilde{\alpha}_{(g,\tau)}, \text{ and},
    \end{equation}
  \item as co-oriented contact groupoids, $(G,\alpha,F) = (G_0,\alpha_0,F_0)$.
  \end{itemize}
  We use the notation by
  $(\tilde{G},\tilde{\alpha},\tilde{F}) =
  \{(G_{\tau},\alpha_\tau,F_{\tau})\}$; if $\{G_\tau\}$ is the
  constant deformation, we write $(G \times I,
  \tilde{\alpha},\tilde{F}) = \{(G,\alpha_\tau,F_\tau)\}$.
\end{defn}

\begin{rmk}\label{rmk:deformations_co-oriented}
  A Lie groupoid homomorphism $\tilde{F} : \tilde{G} \to \Rr^*$ is
  completely determined by the Lie groupoid homomorphisms $F_\tau:
  G_\tau \to \Rr^*$ for all $\tau$. Moreover, by equation
  \eqref{eq:25}, the 1-form $\tilde{\alpha}$ is completely determined
  by $\alpha_\tau$ for all $\tau$. Therefore, the multiplicative 1-form
  $(\tilde{\alpha},\tilde{F})$ is completely determined by
  $(\alpha_\tau,F_\tau)$ for all $\tau$. This justifies the above notation (cf. Remarks
  \ref{rmk:contrast_defn} and \ref{rmk:strict_mult_distn}).
\end{rmk}

The following result is the `smooth 1-parameter' analog of Proposition
\ref{prop:canonical_form}.

\begin{corollary}\label{cor:strict_strict}
  Let $(\tilde{G},\tilde{H})$ be a deformation of a co-orientable contact
  groupoid $(G,H)$. For any multiplicative 1-form
  $(\tilde{\alpha},\tilde{F})$ for $(\tilde{G},\tilde{H})$, $(\tilde{G},\tilde{\alpha},\tilde{F})$ is a deformation of the
  co-oriented contact groupoid $(G,\alpha_0,F_0)$ such that
  $(\alpha_\tau,F_\tau)$ is a multiplicative contact form for
  $(G_\tau,H_\tau)$ for all $\tau$. Conversely, any
  deformation $(\tilde{G},\tilde{\alpha},\tilde{F})$ of a co-oriented contact groupoid $(G,\alpha,F)$ induces a
  deformation $(\tilde{G}, \tilde{H}:= \ker \tilde{\alpha})$ of the
  co-orientable contact groupoid $(G,H:=\ker\alpha)$ such that
  $(\alpha_\tau,F_\tau)$ is a multiplicative contact form for $(G_\tau,H_\tau)$.
\end{corollary}

\begin{proof}
  Let $(\tilde{G},\tilde{H})$ be a deformation a co-orientable contact
  groupoid $(G,H)$ and let $(\tilde{\alpha},\tilde{F})$ be a multiplicative 1-form for
  $\tilde{H}$. The latter is equivalent to a
  trivialization of $\tilde{L} =
  T\tilde{G}/\tilde{H}|_{\tilde{M}}$. Then Proposition
  \ref{prop:canonical_form} and Corollary \ref{cor:family_forms} (the
  latter applied to the canonical 1-forms composed with the above
  trivialization), imply that $(\alpha_\tau,F_\tau)$ is a multiplicative contact form for
  $(G_\tau,H_\tau)$ for all $\tau$. Moreover, equation \eqref{eq:25} holds by
  equation \eqref{eq:1}. Hence, $(\tilde{G},\tilde{\alpha},\tilde{F})$ is a deformation of the
  co-oriented contact groupoid $(G,\alpha_0,F_0)$ with the desired
  property.

  Conversely, let $(\tilde{G},\tilde{\alpha},\tilde{F})$ be a
  deformation of $(G,\alpha,F)$ and set $\tilde{H} := \ker
  \tilde{\alpha}$. By Proposition \ref{prop:canonical_form} and
  Corollary \ref{cor:family_forms}, 
  $(\alpha_\tau,F_\tau)$ is a multiplicative contact form for
  $(G_\tau,H_\tau)$ for all $\tau$. Moreover, $\tilde{H}$ satisfies equation
  \eqref{eq:1} since $\tilde{\alpha}$ satisfies equation
  \eqref{eq:25}. Hence, $(\tilde{G},\tilde{H})$ is a deformation of
  $(G,H)$ with the desired property.
\end{proof}


In a deformation of a co-oriented contact groupoid, the sign of $F$ is
constant (see Definition \ref{defn:reeb_cocycle}). More precisely, the
following holds.

\begin{lemma}\label{lemma:co-oriented_defo}
  Let $(\tilde{G},\tilde{\alpha},\tilde{F}) =
  \{(G_\tau,\alpha_\tau,F_\tau)\}$ be a deformation of a co-oriented
  contact groupoid $(G,\alpha,F)$. Let $\mathrm{pr} : G \times I \to G$
  denote projection onto the first component and let $\tilde{\sigma}$
  (respectively $\sigma_\tau$) denote the sign of $\tilde{F}$
  (respectively $F_\tau$). Then $\tilde{\sigma} = \mathrm{pr}^*\sigma_0$
  and $\sigma_\tau = \sigma_0$ for all $\tau$.
\end{lemma}
\begin{proof}
  Since $I$ is connected, for any $g \in G$, the map $\tau \mapsto \tilde{\sigma}(g,\tau)$ is
  constant. Hence, $\tilde{\sigma} =
  \mathrm{pr}^*\sigma_0$ and, by Definition \ref{defn:deformation_co-oriented}, $\sigma_\tau =
  i^*_\tau\mathrm{pr}^*\sigma_0 = \sigma_0$ as desired.
\end{proof}

To conclude this section, we prove the following results, which are the `smooth 1-parameter' analogs of Proposition
\ref{prop:proper_co-orientable} and of Corollary
\ref{cor:proper_co-orientable}.

\begin{Prop}\label{familyOfForms}
  Let $(\tilde{G},\tilde{\alpha},\tilde{F})$ be a deformation of a co-oriented
  contact groupoid $(G,\alpha,F)$ and let $\sigma =
  \mathrm{sgn}(F)$. If the cocycle $\tilde{r}$ associated to
  $(\tilde{G},\tilde{\alpha},\tilde{F})$ is a coboundary, then, for all
  $\tilde{\kappa} \in C^{\infty}(\tilde{M})$ satisfying equation \eqref{eq:reeb_trivial}, $(e^{\tilde{t}^*\tilde{\kappa}}\tilde{\alpha},\mathrm{pr}^*\sigma)$
  is a multiplicative 1-form for $(\tilde{G},\tilde{H} = \ker \tilde{\alpha})$, where
  $\mathrm{pr} : \tilde{G} = G \times I \to G$ is projection onto
  the first component.
\end{Prop}

\begin{proof}
  By Lemma \ref{lemma:co-oriented_defo}, the sign of
  $\tilde{F}$ equals $\mathrm{pr}^*\sigma$. By Remark
  \ref{rmk:corank_one_mult_dist} and Proposition \ref{prop:proper_co-orientable}, $(e^{\tilde{t}^*\tilde{\kappa}}\tilde{\alpha},\mathrm{pr}^*\sigma)$
  is a multiplicative 1-form for $(\tilde{G},\tilde{H})$, as desired.
\end{proof}

Proposition \ref{familyOfForms} and vanishing of differentiable
cohomology in positive degrees for proper Lie groupoids (see
\cite[Proposition 1]{crainic}), immediately
imply the following result.

\begin{corollary}\label{cor:existence_form}
  Let $(\tilde{G},\tilde{H})$ be a proper deformation of a co-orientable
  contact groupoid $(G,H)$. Then there exist a Lie groupoid
  homomorphism $\sigma : G \to \{ \pm 1\}$ and $\tilde{\alpha} \in
  \Omega^1(\tilde{G})$ such
  that $(\tilde{\alpha},\mathrm{pr}^*\sigma)$ is a multiplicative
  1-form for $(\tilde{G},\tilde{H})$, where $\mathrm{pr} : \tilde{G} = G \times I \to G$ is
  projection onto the first component.
\end{corollary}

\section{Multiplicative Gray Stability}\label{sec:mult-vers-grays}
In this section we prove our main result, Theorem \ref{thm:gray_dist}. Together with Corollary
\ref{cor:strict_deformations_mult_structures}, it yields stability of
compact contact groupoids over connected manifolds (see Corollary \ref{cor:stability}). The proof of Theorem \ref{thm:gray_dist}
is effectively given by combining Theorems \ref{CompactGray} and
\ref{thm:main_not_triv}. We present the proof in this
fashion to emphasize these other results, which are interesting in
their own right, and to simplify the exposition.

\begin{Theo}\label{thm:gray_dist}
  Let $(G,H)$ be a compact contact groupoid over a connected
  manifold $M$. Any deformation of $(G,H)$ of the form $(G\times I,
  \tilde{H}) = \{(G,H_\tau)\}$ is trivial.
\end{Theo}

By Corollary \ref{cor:strict_deformations_mult_structures}, Theorem
\ref{thm:gray_dist} immediately implies the following stability result.

\begin{corollary}\label{cor:stability}
  Any deformation of a compact contact groupoid over a connected
  manifold is trivial.
\end{corollary}

\begin{proof}[Proof of Theorem \ref{thm:gray_dist}]
   We split the proof in two cases, namely 
   whether $L = TG/H|_M$ is trivializable or not.

   Suppose that $L$ is trivializable. Since $G$ is compact, $G \times
   I$ is proper. Hence, by Corollary \ref{cor:existence_form}, there
   exist a Lie groupoid homomorphism $\sigma : \tilde{G} \to \{ \pm 1\}$
   and $\tilde{\alpha}\in \Omega^1(\tilde{G})$ such that $(\tilde{\alpha},\mathrm{pr}^*\sigma)$ is
   a multiplicative
   1-form for $(\tilde{G},\tilde{H})$, where $\mathrm{pr} : G \times I \to G$
   is projection onto the first component. For all $\tau$, set
   $\alpha_\tau := i^*_\tau \tilde{\alpha}$ and observe that 
   $H_{\tau} = \ker \alpha_\tau$. Then, by Corollary
   \ref{cor:strict_strict},
   $(G \times I,\tilde{\alpha},\mathrm{pr}^*\sigma) = \{(G,
   \alpha_\tau, \sigma)\}$ is a deformation of the co-oriented contact
   groupoid $(G, \alpha_0,\sigma)$. By Theorem \ref{CompactGray}
   below, there exist an automorphism $\tilde{\Phi} = \{\Phi_\tau\}$ of the constant
   deformation of $G$, and $\tilde{f} \in C^{\infty}(M \times I)$
   that is constant along the orbits of $G \times I$, such
   that $ \Phi_{\tau}^*\alpha_\tau =
   e^{t^*f_\tau}\alpha $ for all $\tau$, where $f_{\tau} = i_\tau^*\tilde{f}$. In particular, $\Phi_\tau : (G,H) \to (G,H_\tau)$ is an isomorphism of
   contact groupoids for all
   $\tau$. Hence, by Remark \ref{rmk:iso_deformations},
   $\tilde{\Phi}$ is an isomorphism of deformations of contact
   groupoids between the constant deformation and
   $(\tilde{G},\tilde{H})$, i.e.,  $(\tilde{G},\tilde{H})$ is trivial
   as desired.

   The case in which $L$ is not trivializable is precisely Theorem
   \ref{thm:main_not_triv} below. (This is where we use connectedness of
   $M$.) This completes the proof.
\end{proof}

\begin{Theo}\label{CompactGray} 
  Let $(G,\alpha,\sigma)$ be a compact co-oriented contact groupoid
  over $M$. Let $(G \times I, \tilde{\alpha},\mathrm{pr}^*\sigma) = \{(G,\alpha_\tau,\sigma)\}$ be
  a deformation of $(G,\alpha,\sigma)$. Then there exist an
  automorphism $\tilde{\Phi} = \{\Phi_\tau\}$ of the constant deformation $G \times I$ of $G$, and
  $\tilde{f} \in C^{\infty}(M \times I)$ that is constant along the orbits of
  $G \times I$, such that
  \begin{equation}
    \label{eq:10}
    \Phi_{\tau}^*\alpha_\tau =
    e^{t^*f_\tau}\alpha  
  \end{equation}
  \noindent
  for all $\tau$, where $f_{\tau} = i_\tau^*\tilde{f}$. 
\end{Theo}

\begin{proof}
  We use the standard proof of Gray stability in contact geometry
  (see, e.g., \cite[Theorem 2.2.2]{Geiges}),
  making sure that all choices can be made in a multiplicative fashion. We look for an automorphism $\tilde{\Phi} =
  \{\Phi_\tau\}$ of the constant deformation $G \times I$ and for a
  positive function 
  $\tilde{a} \in C^{\infty}(G \times I)$ such that
  \begin{equation}\label{goal}
    \Phi_{\tau}^*\alpha_\tau=a_\tau\alpha \qquad \text{for all } \tau,
  \end{equation}
  \noindent
  where $a_\tau = i^*_\tau\tilde{a}$. Moreover, we assume that $\Phi_\tau$ is the flow of a time 
  dependent vector field $X_\tau\in\XX(G)$. Differentiating
  \eqref{goal} with respect to $\tau$, we obtain 
  \begin{equation}\label{rewrite}
    \Phi_{\tau}^*\Big{(}\frac{d\alpha_\tau}{d\tau}+\Lie_{X_\tau}\alpha_\tau\Big{)}
    =\frac{da_\tau}{d\tau}\frac{1}{a_\tau}\Phi_{\tau}^*\alpha_\tau. 
  \end{equation}
  \noindent
  Setting $\mu_\tau:=(\frac{d}{d\tau}\log
  a_\tau)\circ\Phi_{\tau}^{-1}$, equation \eqref{rewrite} is
  satisfied if and only if
  \begin{equation}\label{eq:12}
    \frac{d\alpha_\tau}{d\tau}+\Lie_{X_\tau}\alpha_\tau-\mu_\tau\alpha_\tau
    = 0.
  \end{equation}
  \noindent
  For all $\tau$, we look for a solution of \eqref{eq:12} of the form $X_\tau\in\Gamma(H_\tau)$, where $H_\tau = \ker \alpha_\tau$. Then equation \eqref{eq:12}
  reduces to
  \begin{equation}\label{theEq}
    \frac{d\alpha_\tau}{d\tau}+\iota_{X_\tau}d\alpha_\tau=\mu_\tau\alpha_\tau.
  \end{equation}
  \noindent
  Due to the defining relations \eqref{eq:Reeb_eqns}, evaluating 
  $R_{\tau}: = R^{\alpha_\tau}$ in \eqref{theEq} yields 
  \begin{equation*}\label{eq:14}
    \mu_\tau = \frac{d\alpha_\tau}{d\tau}(R_\tau), 
  \end{equation*}
  \noindent
  which determines $\mu_\tau$ uniquely. Moreover, from the decomposition \eqref{eq:7}, 
  $\mu_\tau\alpha_\tau - \frac{d\alpha_\tau}{d\tau}$ is a
  section of $\mathrm{Ann}(\langle R_\tau \rangle)$, thus implying that
  there is a unique solution $X_{\tau} \in \Gamma(H_\tau)$ of \eqref{theEq} depending smoothly on
  $\tau$.

  Suppose that $X_\tau$ is a
  multiplicative vector field and that
  \begin{equation}
    \label{eq:15}
    \mu_\tau(g)=\mu_\tau(1_{s(g)})=\mu_\tau(1_{t(g)})
  \end{equation}
  \noindent
  for all $\tau$ and for all $g \in G$.
  Then the flow $\Phi_\tau$ of $X_\tau$ is a Lie groupoid
  automorphism of $G$. Moreover, by compactness of $G$, the flow of
  $X_\tau$ exists for all $\tau$. Hence, we obtain an automorphism
  $\tilde{\Phi} = \{\Phi_\tau\}$ of the constant deformation of $G$
  satisfying equation \eqref{goal} for all $\tau$. Since $\mu_\tau$
  satisfies equation \eqref{eq:15} for all $\tau$, so does
  $a_\tau$. This is equivalent to $\tilde{a}$ satisfying the analog of
  equation \eqref{eq:15} for the groupoid $G \times I$. Since $\tilde{a}$ is positive by assumption,
  it follows that there exists $\tilde{f} \in C^{\infty}(M \times I)$ constant along the orbits
  of $G \times I$ such that $\tilde{a} =
  e^{\tilde{t}^*\tilde{f}}$. Hence, equation \eqref{eq:10} holds for
  all $\tau$.

  To complete the proof, we need to show that $X_\tau$ is a
  multiplicative vector field and that $\mu_\tau$
  satisfies \eqref{eq:15} for all $\tau$. To this end, we observe that
  $\frac{d\alpha_\tau}{d\tau}\in\Omega^1(G)$ is
  $\sigma$-multiplicative for all $\tau$. Moreover, since $d\sigma \equiv 0$, 
  $d\alpha_\tau\in\Omega^2(G)$ is also $\sigma$-multiplicative. This
  allows us to prove that, for all $\tau$, $\mu_\tau$ satisfies
  \eqref{eq:15}. Indeed, since $\frac{d\alpha_\tau}{d\tau}$ is
    $\sigma$-multiplicative and $R_\tau$ is right-invariant (see Lemma
    \ref{lemma:reeb_gpd}), we have that
    \begin{equation}
      \label{eq:26}
      \begin{split}
        \mu_\tau(g) & =\Big{(}\frac{d\alpha_\tau}{d\tau}\Big{)}_g(R_\tau)=\Big{(}\frac{d\alpha_\tau}{d\tau}\Big{)}_{1_{t(g)}g}(dm(R_\tau,0)) \\ 
        & =\Big{(}\frac{d\alpha_\tau}{d\tau}\Big{)}_{1_{t(g)}}(R_\tau)+\sigma(1_{t(g)})\Big{(}\frac{d\alpha_\tau}{d\tau}\Big{)}_g(0)=\mu_\tau(1_{t(g)})
      \end{split}
    \end{equation}
    \noindent
    for all $\tau$ and for all $g \in G$. Moreover, by Lemma
    \ref{lemma:reeb_gpd}, $
    R_{\tau,g} = dm(0_g,\sigma(g)R_{\tau, 1_{s(g)}})$ for all $\tau$ and for all $g \in G$. Hence, a
    computation entirely analogous to \eqref{eq:26} shows that $\mu_\tau(g) =
    \mu_\tau(1_{s(g)})$ for all $\tau$ and for all $g \in G$, as desired.

    It remains to prove that $X_\tau$ is multiplicative, i.e., the
    map $X_\tau : G \to H_\tau$ is a Lie groupoid homomorphism for all
    $\tau$. Since $\mu_\tau$ satisfies \eqref{eq:15} and $\alpha_\tau$ is
    $\sigma$-multiplicative for all $\tau$, the 1-form $\mu_\tau
    \alpha_\tau$ is also $\sigma$-multiplicative for all $\tau$.
    Hence, 
    $\beta_\tau:=\mu_\tau\alpha_\tau-\frac{d\alpha_\tau}{d\tau}$ 
  is $\sigma$-multiplicative. For each $\tau$,
  $X_\tau$ is the composition
  \begin{equation}
    \label{eq:13}
    \xymatrix{ G \ar[r]^-{\beta_\tau} & T^*G \ar[r] & \H_\tau^* \ar[r]^-{\cong}& \H_\tau, }
  \end{equation}
  \noindent
  where the middle map is the dual to the inclusion and the last is
  the inverse of the restriction of
  $(d\alpha_\tau)^\flat$ to $\H_\tau$. By Lemmas \ref{lemma:aux_gpd} and 
  \ref{lemma:sig-mult} in Appendix \ref{sec:sigma-mult-forms} all maps
  in \eqref{eq:13} are Lie groupoid
  homomorphisms\footnote{The Lie groupoid structures on $T^*G$
    and $H^*_\tau$ that we consider in the above argument are not the
    standard ones. In fact, the ones we consider are obtained by
    modifying 
    the standard ones appropriately using $\sigma$, as explained in
    Appendix \ref{sec:sigma-mult-forms}.} and, therefore, so is $X_\tau$, as desired.
\end{proof} 

\begin{theorem}\label{thm:main_not_triv}
  Let $(G,H)$ be a compact contact groupoid over a connected manifold $M$
  such that $L = TG/H|_M$ is not trivializable. Any deformation
  of $(G,H)$ of the form $(G\times I,
  \tilde{H}) = \{(G,H_\tau)\}$ is trivial.
  
\end{theorem}

\begin{proof}
  Let $\mathrm{pr}: M \times I \to M$ be projection onto the first
  component and let $\tilde{L} = T(G\times I)/\tilde{H}|_{M\times
    I}$. Fix an isomorphism $\psi : \tilde{L} \to \mathrm{pr}^*L$ (see equation~\eqref{eq:9});
  this identifies $L_\tau$ with $L$ for all $\tau$. Consider the
  co-orientable finite cover $(\tilde{G}', \tilde{H}')$ of $(\tilde{G},\tilde{H})$ (see Remark
  \ref{rmk:co-orientable_cover_corank_one}), and let
  $(\hat{G},\hat{H})$ be the co-orientable finite cover of
  $(G,H)$. Then the choice of $\psi$ allows us to identify
  $\tilde{G}'$ with $\hat{G} \times I$. Under this identification, the
  natural $\Zz_2
  \times \Zz_2$-action on $\hat{G}$ becomes the following: for all
  $(a,b) \in \Zz_2 \times \Zz_2$,
  \begin{equation}
    \label{eq:18}
    (a,b) \cdot (\hat{g},\tau) := (\varphi_{ab}(\hat{g}),\tau),
  \end{equation}
  \noindent 
  where $\varphi_{ab}$ is the action of $(a,b)$ on $\hat{G}$. 
  Moreover, $(\hat{G}
  \times I, \tilde{H}') = \{(\hat{G},H'_{\tau})\}$ is a deformation of
  $(\hat{G},\hat{H})$ with the property that
  $(\hat{G},H'_{\tau})$ is the co-orientable finite cover
  $(\hat{G},\hat{H}_\tau)$ of
  $(G,H_{\tau}) $ for all $\tau$. Since $(\hat{G},\hat{H})$ is co-orientable by
  construction, and compact by Remark
  \ref{lemma:compactness_preserved}, we can argue as
  in the proof of Theorem
  \ref{thm:gray_dist}, i.e., there exists a Lie groupoid homomorphism
  $\hat{\sigma}:\hat G\to\{ \pm 1\}$ and $\tilde{\alpha}' \in
  \Omega^1(\tilde{G}')$ such
  that 
  $(\tilde{\alpha}', \mathrm{Pr}^*\hat{\sigma})$ is a multiplicative
  1-form for $(\tilde{G}',\tilde{H}')$,
  where $\mathrm{Pr}: \hat{G} \times I \to \hat{G}$ is projection onto
  the first component. Set
  $\hat{\alpha}_\tau:=\hat{i}_\tau^*\tilde{\alpha}'$, where
  $\hat{i}_\tau : \hat{G} \to \hat{G} \times I$,
  $\hat{i}_\tau(\hat{g}) = (\hat{g},\tau)$. Then we obtain a
  deformation $(\hat{G} \times I, \tilde{\alpha}',
  \mathrm{Pr}^*\hat{\sigma})$ of the co-oriented contact groupoid
  $(\hat{G},\hat{\alpha}_0,\hat{\sigma})$. Hence, Theorem
  \ref{CompactGray} can be applied. In fact, suppose that the
  time-dependent vector field constructed in the proof of Theorem
  \ref{CompactGray} is $\Zz_2 \times
  \Zz_2$-equivariant. Then the automorphism of
  the constant deformation $\hat{G} \times I$ in Theorem
  \ref{CompactGray} induces an automorphism of
  the constant deformation $G \times I$. Since the former is an
  isomorphism of deformations of contact groupoids between
  the constant deformation of $(\hat{G},\hat{H})$ and $(\hat{G} \times
  I, \tilde{H}')$, it follows that the latter is the desired
  isomorphism of deformations of contact groupoids between the
  constant deformation of $(G,H)$ and $(\tilde{G},\tilde{H})$.

  It remains to prove the time-dependent vector field constructed in
  the proof of Theorem \ref{CompactGray} is $\Zz_2 \times \Zz_2$-equivariant. We fix
  notation as in the proof of Theorem \ref{CompactGray}, adding
  hats. We need to prove that the unique solution $\hat{X}_\tau \in
  \Gamma(\hat{H}_\tau)$ to
  \begin{equation}
    \label{eq:17}
    \frac{d\hat{\alpha}_\tau}{d\tau}+\iota_{\hat{X}_\tau}d\hat{\alpha}_\tau=\hat{\mu}_\tau\hat{\alpha}_\tau 
  \end{equation}
  \noindent
  is $\Zz_2 \times \Zz_2$-equivariant. Since the $\Zz_2 \times
  \Zz_2$-action on $\hat{G} \times I$ is given as in equation
  \eqref{eq:18}, this is equivalent to showing that
  \begin{equation}
    \label{eq:19}
     \hat{X}_\tau \underset{\varphi_{ab}}{\sim} \hat{X}_\tau \qquad
     \text{for all } \tau \text{ and for all } (a,b) \in \Zz_2\times\Zz_2.
  \end{equation}
  \noindent
  We observe that $d\varphi_{ab}(\hat{H}_\tau) = \hat{H}_\tau$ for all $\tau$ and for all $(a,b) \in \Zz_2
  \times \Zz_2$. Hence, by
  uniqueness of $\hat{X}_\tau$, in order to prove \eqref{eq:19}, it
  suffices to show that $d\varphi_{ab}(\hat{X}_\tau)$ solves
  equation \eqref{eq:17} for all $\tau$ and for all
  $(a,b) \in \Zz_2 \times \Zz_2$. To this end, let $f_{\tau,ab} \in C^{\infty}(\hat{G})$ be the
  nowhere vanishing function satisfying
  \begin{equation}
    \label{eq:22}
     \varphi_{ab}^*\hat{\alpha}_\tau=f_{\tau,ab}\hat{\alpha}_\tau. 
  \end{equation}
  \noindent
  Since $\hat{G}$ is compact, we can apply Lemma
  \ref{lemma:the_fs_are_constant} to $(\hat{G},\hat{H}_\tau)$ for all
  $\tau$, thus obtaining that
  \begin{equation}
    \label{eq:20}
    f_{\tau,00} = f_{\tau,01} \equiv 1 \qquad  f_{\tau,10} =
    f_{\tau,11} \equiv -1 \qquad \text{for all } \tau.
   \end{equation}
  \noindent
  In particular, $f_{\tau,ab}$
  does not depend on $\tau$ for all $(a,b) \in \Zz_2 \times \Zz_2$. For this reason, we drop the notational
  dependence on $\tau$. Equation \eqref{eq:20} implies that
  $\varphi_{ab}^*\frac{d\alpha_\tau}{d\tau}=f_{ab}\frac{d\alpha_\tau}{d\tau}$
  and $\varphi_{ab}^*d\alpha_\tau = f_{ab}d\alpha_\tau$ for all
  $\tau$ and for all
  $(a,b) \in \Zz_2 \times \Zz_2$. Therefore,
  combining equations \eqref{eq:17}, \eqref{eq:22} and \eqref{eq:20}, we have that
  \begin{equation}
    \label{eq:21}
    \iota_{d\varphi_{ab}(\hat{X}_{\tau})}
  (d\hat{\alpha}_\tau)
  = (\varphi_{ab}^*\hat{\mu}_\tau)\hat{\alpha}_\tau-\frac{d\hat{\alpha}_\tau}{d\tau},
  \end{equation}
  \noindent
  for all $\tau$ and for all $(a,b) \in \Zz_2 \times \Zz_2$.
  Suppose that $\varphi_{ab}^*\hat{\mu}_\tau = \hat{\mu}_\tau$ for all $\tau$ and for all $(a,b) \in \Zz_2 \times
  \Zz_2$. Then equation \eqref{eq:21} implies
  the desired result, since both 
  $d\varphi_{ab}(\hat{X}_{\tau})$ and $\hat{X}_\tau$ are sections of
  $H_\tau$ for all $\tau$.

  It remains to show that $\varphi_{ab}^*\hat{\mu}_\tau =
  \hat{\mu}_\tau$ for all $\tau$ and for all $(a,b) \in \Zz_2 \times
  \Zz_2$. By equation \eqref{eq:15}, for all $\tau$, there exists
  $\hat{\delta}_\tau \in C^{\infty}(\hat{M})$ such that $\hat{\mu}_\tau =
  \hat{s}^*\hat{\delta}_{\tau} =
  \hat{t}^*\hat{\delta}_{\tau}$. Moreover, $\hat{\delta}_\tau$ is
  constant along the orbits of $\hat{G}$ for all $\tau$. If
  $\gamma$ denotes the action of the non-trivial element of $\Zz_2$ on
  $\hat{M}$, then by equation \eqref{eq:23} we have that
  \begin{equation}
    \label{eq:28}
    \varphi^*_{ab}\hat{\mu}_{\tau} =
  \hat{t}^*(\gamma^a)^*\hat{\delta}_\tau =
  \hat{s}^*(\gamma^b)^*\hat{\delta}_\tau
  \end{equation}
  \noindent
  for all $\tau$ and for all $(a,b) \in
  \Zz_2 \times \Zz_2$.
  Using equations \eqref{eq:24} and \eqref{eq:23}, we see that, for all $x \in
  \hat{M}$, $\gamma^c(x)$ lies in the same
  orbit of $x$ for all $c \in \Zz_2$. Hence, since
  $\hat{\delta}_\tau$ is constant along the orbits of $\hat{G}$,
  $(\gamma^c)^*\hat{\delta}_\tau = \hat{\delta}_{\tau}$ for all $\tau$ and for all $c \in \Zz_2$. Therefore,
  equation \eqref{eq:28} implies that $\varphi_{ab}^*\hat{\mu}_\tau =
  \hat{\mu}_{\tau}$ for all $\tau$ and for all
  $(a,b) \in \Zz_2 \times \Zz_2$, as desired. 
\end{proof}

\section{Multiplicative Gray stability at the level of
  objects}\label{sec:examples-&-apps}
This section has three aims: to define Jacobi bundles (see Definition \ref{defn:jac_str_manifolds}), to recall that
they appear on the base manifolds of contact groupoids (see Theorem \ref{LocToGlob}), and to illustrate how Corollary
\ref{cor:stability} can be applied to their deformations (see Theorems \ref{thm:gray_Jacobi} and
\ref{thm:gray_poisson}). Moreover, we prove that Jacobi bundles that are induced by proper
co-orientable contact groupoids `are Poisson' (see
Corollary \ref{cor:proper_jacobi}). The main references are \cite{crainic_salazar,crainic_zhu,dazord,kerbrat,zamb_zhu}.

\subsection{Jacobi bundles}\label{subsec:jac-mfds}

\begin{defn}\label{defn:jac_str_manifolds}
  Let $L \to M$ be a line bundle. A {\bf Jacobi bracket on
    $\boldsymbol{L \to M}$} is a local 
  Lie bracket on  $\Gamma(L)$, i.e., for all 
  $u,v \in \Gamma(L)$, $ \mathrm{supp}\left(\{u,v\}\right) \subset
  \mathrm{supp}(u) \cap \, 
  \mathrm{supp}(v)$. The triple $\jm$ is a {\bf Jacobi bundle}.
\end{defn}

\begin{example}[Contact manifolds]\label{example:contact_manifolds}
  Let $(N,H)$ be a contact manifold. A vector field 
  $X \in \XX(N)$ is {\em Reeb} if 
  $[X,\Gamma(\H)]\subset\Gamma(\H)$. By the Jacobi identity, 
  the subspace of Reeb vector fields is a Lie subalgebra of 
  $\XX(N)$ that we denote by
  $\XX_{\mathrm{Reeb}}(N,\H)$. Moreover, the generalized contact form
  $\alpha_{\mathrm{can}}$ induces an isomorphism of vector spaces
  $\XX_{\mathrm{Reeb}}(N,\H) \cong \Gamma(L)$. Since the Lie bracket of vector 
  fields is local, the Lie algebra structure on 
  $\XX_{\mathrm{Reeb}}(N,\H)$ induces a Jacobi bracket
  $\{\cdot,\cdot\}_H$ on $L \to N$.
\end{example}

\begin{example}[Poisson manifolds]\label{example:Poisson}
  Let $(M,\pi)$ be a {\bf Poisson manifold}, i.e., $\pi \in \XX^2(M)$
  such that
  \begin{equation}
    \label{eq:30}
    \llbracket \pi, \pi \rrbracket =0,
  \end{equation}
  where $\llbracket \cdot, \cdot \rrbracket$ is the 
  Schouten-Nijenhuis bracket. The condition 
  \eqref{eq:30} is equivalent to $\{f,g\}_{\pi}:= \pi(df,dg)$ being a
  Lie bracket being a Lie bracket on $C^{\infty}(M)$. Since 
  $\pi$ is a bivector, $\{\cdot,\cdot\}_\pi$ is a Jacobi bracket on $M \times \Rr \to M$.
\end{example}

We say that a Jacobi bundle $\jm$ is {\bf trivializable} if $L \to M$
is trivializable. A trivialization $\psi
: L \to M \times \Rr$ induces a Jacobi bracket $\{\cdot,\cdot\}$
on $M \times \Rr \to M$. In \cite{kirillov} it is shown that 
there exists a unique pair $(\Lambda,R) \in \XX^2(M) \times \XX(M)$ such that, for all 
$f,g \in C^{\infty}(M)$, 
\begin{equation}\label{eqn:jac_trivial}
  \lbrace f,g\rbrace =\Lambda(df,dg)+fR g -gR f.
  \end{equation}
\noindent
If $\psi'$ is another trivialization of $L$, we let $\Lambda' \in
\XX^2(M)$, $R' \in \XX(M)$ be as in equation \eqref{eqn:jac_trivial}
for the induced Jacobi bracket on $M \times \Rr \to M$.
We identify $\psi' \circ \psi^{-1}$ with a nowhere vanishing $a
\in C^{\infty}(M)$. For any pair $(\Lambda, R) \in \XX^2(M) \times
\XX(M)$, set
\begin{equation}
  \label{eq:29}
  \Lambda^a:=a\Lambda \qquad \text{and} \qquad
  R^a:=aR +\Lambda^\sharp(da).
\end{equation}
\noindent
Then we have that
\begin{equation}
  \label{eq:33}
  (\Lambda',R') = (\Lambda^{a^{-1}}, R^{a^{-1}}).
\end{equation}

\begin{definition}\label{defn:co-oriented_jac}
  Let $\jm$ be a trivializable Jacobi bundle. If $(\Lambda,R) \in
  \XX^2(M) \times \XX(M)$ satisfies equation \eqref{eqn:jac_trivial}
  (for some
  trivialization $\psi$ of $L$) , then we call $(\Lambda,R)$
  a {\bf Jacobi pair} (for $\jm$). The triple $(M,\Lambda,R)$ is a {\bf
    Jacobi manifold}.
\end{definition}

\begin{Rmk}\label{rmk:co-oriented_jac}
  If $(M,\Lambda, R)$ is a Jacobi manifold, then 
  \begin{equation}\label{JacobiPair}
    \llbracket \Lambda,\Lambda \rrbracket =2R\wedge\Lambda \qquad
    \text{and} \qquad \llbracket \Lambda,R \rrbracket=0.
  \end{equation}
  \noindent
  Conversely, any pair $(\Lambda, R) \in \XX^2(M) \times \XX(M)$ 
  satisfying \eqref{JacobiPair} yields a Jacobi bundle 
  $(M, M \times \mathbb{R},\{\cdot,\cdot\})$, where $\{\cdot,\cdot\}$
  is defined by the right hand side of 
  \eqref{eqn:jac_trivial}. 
\end{Rmk}

\begin{example}[Co-orientable contact manifolds]\label{exm:cooriented_contact}
  Let $(N,H)$ be a co-orientable contact manifold. A contact form
  $\alpha$ for $(N,H)$ induces a Jacobi pair for the Jacobi bundle
  $(N, L= TN/H, \{\cdot,\cdot\}_H)$ of Example
  \ref{example:contact_manifolds} as follows. There exists a unique 
  bivector field\footnote{The notation used in this Example should not
  be confused with the notation of equation \eqref{eq:29}. Here we
  merely wish to indicate that the bivector field and the Reeb vector
  field depend on $\alpha$.} $\Lambda^{\alpha} \in \XX^2(N)$ such that 
  $\Lambda^{\alpha}(\alpha,-)= 0$ and 
  $\Lambda^{\alpha}|_{H^* \times H^*} = (d\alpha|_{H \times H})^{-1}$. 
  If $R^{\alpha}$ is the 
  Reeb vector field of $\alpha$, then $(\Lambda^{\alpha},R^{\alpha})$
  is a Jacobi pair for $(N, L, \{\cdot,\cdot\}_H)$.

  If $\alpha'$ is another contact form for $(N,H)$, then there exists
  a nowhere vanishing $a \in C^\infty(M)$ such that $\alpha' =
  a\alpha$. Then we have that
  \begin{equation*}
    (\Lambda^{\alpha'}, R^{\alpha'}) =
     ((\Lambda^{\alpha})^{a^{-1}},(R^{\alpha})^{a^{-1}}),
   \end{equation*}
   \noindent
   (see equations \eqref{eq:29} and \eqref{eq:33}).
\end{example}

\begin{example}[Poisson manifolds as Jacobi manifolds]\label{exm:Poisson_as_Jacobi}
  Every Poisson manifold $(M,\pi)$ is a Jacobi manifold with $R \equiv
  0$. For this reason, we denote a Poisson manifold viewed as a Jacobi
  manifold simply by $(M,\pi)$.
\end{example}

\begin{defn}\label{defn:jac_maps}
  Let $\left(M_j,L_j, \bracket_j\right)$ be a 
  Jacobi bundle for $j=1,2$. A {\bf Jacobi map between 
  $\boldsymbol{\left(M_1,L_1, \bracket_1\right)}$ and 
  $\boldsymbol{\left(M_2, L_2, \bracket_2\right)}$} is a pair 
  $\left(\phi, B\right)$, where $\phi : M_1 \to M_2$ is
  smooth and $B : \phi^*L_2 \to L_1$ is an isomorphism of 
  line bundles called the {\bf bundle component} such that, for all $u,v \in \Gamma(L_2)$,
  $$ B\circ \phi^*\left\{u, v\right\}_2 =  \left\{B \circ \phi^*u, B \circ \phi^*v\right\}_1.$$
  \noindent
  An {\bf isomorphism} between Jacobi bundles is a Jacobi 
  map $\left(\phi, B\right)$ with the property that $\phi$ is 
  a diffeomorphism.
\end{defn}

\begin{remark}\label{rmk:jacobi_map}
  A Jacobi map between the Jacobi manifolds $\left(M_1,\Lambda_1,
    R_1\right)$ and $\left(M_2,\Lambda_2,
    R_2\right)$ is given by a smooth $\phi : M_1 \to M_2$ and a
  nowhere vanishing $a \in C^{\infty}(M_1)$ (playing the role of the
  bundle component), such that
  $$ \phi_*(\Lambda_1^a)=\Lambda_2,\ \text{ and } \
  \phi_*(R_1^a)=R_2, $$
  \noindent
  (see equation \eqref{eq:29}). As in Definition \ref{defn:jac_maps},
  we denote this map by $(\phi,a)$. In particular, if $(\phi,a)$ is a Jacobi isomorphism between Poisson manifolds 
  $(M_1,\pi_1)$, $(M_2,\pi_2)$, then 
  \begin{equation}\label{eq:conf_iso}
    \phi_*(a\pi_1)=\pi_2 \ \text{ and }\ \pi_1^\sharp(da)=0.
  \end{equation}
\end{remark}

\subsection{From contact groupoids to Jacobi bundles}\label{sec:Jacobi_manifolds} 
In this section we discuss Jacobi bundles that are {\it integrable} by
contact groupoids with particular emphasis on the proper and
co-orientable case. We start by recalling one of the main results of \cite{crainic_salazar}.

\begin{Theo}[Theorem 1 in \cite{crainic_salazar}]\label{LocToGlob}
  Let $(G,\H)$ be a contact groupoid over $M$ and let $L :=
  (TG/\H)|_{M}$. There exists a unique Jacobi bracket
  $\{\cdot,\cdot\}$ on $L \to M$ such that $t:G\to M$ is a
  Jacobi map with bundle component $\mathtt{r}:t^*L \to TG/\H$  (see \eqref{eq:right_iso}).
\end{Theo}

We say that a contact groupoid $(G,H)$ {\bf induces} the Jacobi
bracket $\{\cdot,\cdot\}$ on $L \to M$ of Theorem \ref{LocToGlob},
that it {\bf integrates} $\jm$, and that $\jm$ is {\bf integrable} by
$(G,H)$. We use the same terminology for co-oriented contact groupoids
integrating Jacobi manifolds (see Remark \ref{eq:32} below).

\begin{remark}\label{eq:32}
When $(G,\alpha,F)$ is a co-oriented contact groupoid, 
  the induced Jacobi pair $(\Lambda,R)$ on $(M,M \times \Rr,\{\cdot,\cdot\})$ is given by
  $\Lambda=t_*(\Lambda^\alpha)$, $ R=t_*(R^\alpha)$, 
  where $\Lambda^\alpha, R^\alpha$ are as in 
  Example~\ref{exm:cooriented_contact}. Hence, $(t,1)$ is a Jacobi
  map between $(G,\Lambda^{\alpha},R^{\alpha})$ and $(M, \Lambda,R)$
  -- see Remark \ref{rmk:jacobi_map}.
\end{remark}


\begin{remark}\label{rmk:contact_jacobi}
  Let $\Phi : (G_1,H_1) \to (G_2,H_2)$ be an isomorphism of contact
  groupoids covering $\phi : M_1 \to M_2$. Then $(\phi,B):(M_1,L_1,\{\cdot,\cdot\}_1)\to(M_2,L_2,\{\cdot,\cdot\}_2)$
  is a Jacobi isomorphism, where $B$ is the inverse of $$TG_1/H_1|_{M_1}\to
  \phi^*(TG_2/H_2|_{M_2}),\ \ [X]\mapsto (g,[d_g\Phi(X)]).$$ 
\end{remark}

The following result provides a useful characterization of co-oriented
contact groupoids integrating Poisson manifolds.

\begin{lemma}\label{lemma:constant-F} Let $(G,\alpha, F)$ be a
  co-oriented contact groupoid integrating the Jacobi
  manifold $(M, \Lambda, R)$. If $dF \equiv 0$ then $R \equiv 0$, i.e., $\Lambda$
  is a Poisson bivector. Conversely, if $R \equiv 0$ then $dF \equiv 0$ on $\ker
  ds.$ 
\end{lemma}

\begin{proof}
  Let $R^{\alpha,\mathtt{L}}=FR^\alpha+\Lambda^{\alpha,\sharp}(dF)$ be the left-invariant vector 
  field induced by $R^{\alpha}$ (see Lemma \ref{lemma:reeb_gpd} and equation \eqref{eq:16}). Fix $x \in M$. Then $F(x) = 1$ and
  $R^{\alpha,\mathtt{L}}_x=R^\alpha_x -
  dt(R^{\alpha}_x)$ -- see equation \eqref{eq:8}. Hence, 
  $\Lambda^{\alpha,\sharp}(dF)_x = -dt(R^\alpha_x)$ is tangent to
  $M$. By Remark \ref{eq:32}, it follows that 
  \begin{equation}\label{eq:equiv}
  \Lambda^{\alpha,\sharp}(dF)_x=-R_x.
  \end{equation}
  \noindent
  
  By equation \eqref{eq:equiv}, if $dF \equiv 0$ then $R \equiv 0$. 
  Conversely, suppose that $R \equiv 0$. Then, by equation \eqref{eq:equiv}, 
  $\Lambda^{\alpha,\sharp}(dF)|_M \equiv 0$. This means that $dF
  \equiv 0$ 
  on $(\ker\alpha)|_M$. On the other hand, we have that $0 = d\alpha(\Lambda^{\alpha,\sharp}(dF), R^{\alpha}) =
  -dF(R^{\alpha})$.
  Since $TG|_M=\Rr \langle R^\alpha|_M \rangle \oplus\ker\alpha|_M$ (see
  equation \eqref{eq:6}), 
  then $dF$ is zero on $TG|_M$ and, in particular, on 
  $\ker ds|_M$. Since $F: G \to \Rr^*$ is a Lie groupoid homomorphism, it follows that 
  $dF \equiv 0$ on $\ker ds$.
\end{proof}

To conclude this section, we prove the analogs of Proposition
\ref{prop:proper_co-orientable} and Corollary
\ref{cor:proper_co-orientable} for Jacobi bundles that are integrable
by proper co-orientable contact groupoids. First, we recall the following facts.
\begin{enumerate}[label=(\alph*), ref=(\alph*),leftmargin=*]
\item \label{item:1} The base of a Jacobi bundle comes endowed with a {\em singular}
  foliation (see \cite{dual_pairs,kirillov}). In the case of a Poisson
  manifold, this coincides with the symplectic foliation. When $\jm$
  is integrable by $(G,H)$, the leaves of the foliation coincide with
  the connected components of the orbits of $G$.
\item \label{item:2} If $(M,\pi)$ is a 
  Poisson manifold, a function $f\in C^\infty(M)$ is a 
  {\bf Casimir} of $(M,\pi)$ if $\pi^{\sharp}(df) \equiv 0$ (cf. the second condition in equation
  \eqref{eq:conf_iso}).   If $(M,\pi)$ is integrable by a co-oriented contact groupoid, Casimirs of 
  $(M,\pi)$ are precisely smooth functions that are constant 
  along connected components of the orbits of the groupoid. 
\end{enumerate}

\begin{Prop}\label{prop:evenDimLeaves}
  Let $(M,\Lambda, R)$ be a
  Jacobi manifold integrable by a co-oriented contact groupoid
  $(G,\alpha, F)$. Suppose that the Reeb cocycle $r$ of $(G,\alpha, F)$ is
  a coboundary. Then
  \begin{enumerate}[leftmargin=*]
  \item \label{item:5} all leaves of the foliation of $(M,\Lambda,R)$
  are even dimensional.
  \end{enumerate}
  Moreover, for all $\kappa \in C^{\infty}(M)$ satisfying
  equation \eqref{eq:reeb_trivial}, 
  \begin{enumerate}[leftmargin=*, start = 2]
  \item\label{part_2} $\Lambda^{e^{-\kappa}}=e^{-\kappa}\Lambda$ is a
    Poisson bivector, 
  \item\label{propty_3} if $\sigma=\mathrm{sgn}(F)$, then $(G,e^{t^*\kappa}\alpha,\sigma)$ integrates the Poisson manifold $(M,e^{-\kappa}\Lambda)$, and
 \item\label{propty_4} if $R \equiv 0$ then $\kappa$ is a Casimir of the Poisson
  manifold $(M,\Lambda)$.
    \end{enumerate}
\end{Prop}

\begin{proof}
  Set $H:=\ker \alpha$. By Proposition
  \ref{prop:proper_co-orientable}, for any $\kappa \in
  C^{\infty}(M)$ satisfying equation \eqref{eq:reeb_trivial}, $(e^{t^*\kappa}\alpha,\sigma)$ is a
  multiplicative contact form for $(G,H)$. Hence, by Lemma
  \ref{lemma:constant-F}, the Jacobi bracket
  on $M$ induced by $(G,e^{t^*\kappa}\alpha,\sigma)$ is, in 
  fact, a Poisson structure. By Remarks ~\ref{rmk:jacobi_map} and \ref{eq:32}, the corresponding Poisson bivector is given by
  $$t_*(\Lambda^{e^{t^*\kappa}\alpha})=t_*(e^{-t^*\kappa}\Lambda^\alpha)=e^{-\kappa} t_*(\Lambda^\alpha)=e^{-\kappa}\Lambda.$$ 
  \noindent
  This proves parts \ref{part_2} and \ref{propty_3}. Moreover, by
  \ref{item:1}, the orbits of $G$ are even
  dimensional and so are the leaves of $(M,\Lambda,R)$. This shows
  part \ref{item:5}. For part \ref{propty_4}, 
if $R \equiv 0$, then by Lemma \ref{lemma:constant-F}, 
  $F=\sigma e^r = \sigma e^{s^*\kappa-t^*\kappa}$ is constant along the connected 
  components of the source fibers. Hence, since $\sigma$ is locally constant, $\kappa$ is constant 
  on the connected components of the orbits $S_x=t(s^{-1}(x))\subset M$. 
  Hence, by \ref{item:2} above, $\kappa$ is a Casimir of $\Lambda$.  
\end{proof}

Since the differentiable cohomology of proper Lie groupoids vanishes in
all positive degrees (see 
\cite[Proposition $1$]{crainic}), Proposition \ref{prop:evenDimLeaves} immediately implies the
following result. 

\begin{corollary}\label{cor:proper_jacobi}
  Let $\jm$ be a trivializable Jacobi bundle that is integrable by a proper co-orientable contact
  groupoid $(G,H)$. Then there exists a Poisson bivector $\pi \in
  \XX^2(M)$ such that $(\pi,0)$ is a Jacobi pair for $\jm$.
\end{corollary}

\subsection{Stability results for integrable deformations of Jacobi bundles and Poisson structures}\label{sec:at-level-objects-1}
By Theorem \ref{LocToGlob}, given a deformation 
$(\tilde{G},\tilde{H}) = \{(G_\tau,H_\tau)\}$ of a contact groupoid
$(G,H)$, $(G_\tau,H_\tau)$ induces a Jacobi bundle
$(M,L_\tau,\{\cdot,\cdot\}_\tau)$ for each $\tau$.
This motivates introducing the following notion.

\begin{defn}\label{defn:deformation_objects}
  An {\bf integrable deformation of a Jacobi bundle
    $\boldsymbol{\jm}$} is
  \begin{itemize}[leftmargin=*]
  \item a contact groupoid $(G,H)$ integrating $\jm$, and
  \item a deformation $(\tilde{G},\tilde{H}) = \{(G_{\tau},H_\tau)\}$
    of $(G,H)$.
  \end{itemize}
  We use the notation $\{(M,L_\tau,\{\cdot,\cdot\}_\tau)\}$, where,
  for each $\tau$, $(G_\tau,H_\tau)$ induces the Jacobi bundle $(M,L_\tau,\{\cdot,\cdot\}_\tau)$.
\end{defn}

Let $(G,H)$ be a co-orientable contact groupoid
integrating a Jacobi bundle $\jm$ and let $(\tilde{G},\tilde{H}) =
\{(G_{\tau},H_\tau)\}$ be a {\em proper} deformation of $(G,H)$. By Corollary
\ref{cor:existence_form} and Lemma \ref{lemma:constant-F}, there exists a `smooth 1-parameter family'
of Poisson bivectors $\{\pi_\tau\}$ on $M$ such that, for all $\tau$,
$(\pi_\tau,0)$ is a Jacobi pair for
$(M,L_\tau,\{\cdot,\cdot\}_\tau)$. This motivates introducing the
following notion.

\begin{defn}\label{defn:deformation_objects_poisson}
  A {\bf contact integrable deformation of a Poisson manifold
    $\boldsymbol{(M,\pi)}$} is
  \begin{itemize}[leftmargin=*]
  \item a co-oriented contact groupoid $(G,\alpha,F)$ integrating
    $(M,\pi)$, and
  \item a deformation $(\tilde{G},\tilde{\alpha}, \tilde{F}) =
    \{(G_\tau,\alpha_\tau,F_\tau)\}$ of $(G,\alpha,F)$ such that, for
    all $\tau$, $(G_\tau,\alpha_\tau,F_\tau)$ induces a Poisson
    bivector $\pi_\tau$ on $M$ and $\pi = \pi_0$.
      \end{itemize}
 We use the notation $\{(M,\pi_\tau)\}$.
\end{defn}

In Definitions \ref{defn:deformation_objects} --
\ref{defn:deformation_objects_poisson}, we say that (a contact) an 
integrable deformation of a Jacobi bundle (respectively Poisson
manifold) is {\bf compact} if the underlying contact groupoid being deformed
  is compact. In view of Remark \ref{rmk:contact_jacobi}, Corollary \ref{cor:stability} implies the following stability
result for compact integrable deformations of Jacobi bundles. 


\begin{Theo}[Jacobi version]\label{thm:gray_Jacobi}
  Let $M$ be a connected manifold. Given any compact integrable deformation $\{(M,L_\tau,\{\cdot ,
  \cdot \}_\tau)\}$ of a Jacobi bundle $\jm$, there exists
   a smooth $1$-parameter family of Jacobi isomorphisms $(\phi_\tau, B_\tau):
   (M,L_\tau,\{\cdot,\cdot\}_\tau) \to \jm$ starting at the identity.
\end{Theo}


Corollary \ref{cor:stability} also provides the following
characterization of compact contact integrable deformations of Poisson manifolds.

\begin{Theo}[Poisson version]\label{thm:gray_poisson}
  Given a compact contact integrable deformation $\{(M,\pi_\tau)\}$ of a connected
  Poisson manifold $(M,\pi)$, there exist
  a diffeotopy $\{\phi_{\tau}\} \subset \mathrm{Diff}(M)$ and a smooth
  $1$-parameter family $\{a_{\tau}\}$ of positive Casimirs of
  $(M,\pi)$ with $a_0 \equiv 1$, 
  such that
  \begin{equation}
    \label{eq:11}
    (\phi_{\tau})_* \pi_\tau =
    a_{\tau}\pi \qquad \text{for all } \tau.
  \end{equation}
\end{Theo}

\begin{proof}
  Let $(\tilde{G},\tilde{\alpha},\tilde{F}) =
  \{(G_\tau,\alpha_\tau,F_\tau)\}$ be a deformation of co-orientable 
  groupoids inducing $\{(M,\pi_\tau)\}$ such that $G = G_0$ is
  compact. Set $\tilde{H} := \ker \tilde{\alpha}$ and $H := \ker
  \alpha_0$. By
  Corollary \ref{cor:stability}, there exists an isomorphism
  $\tilde{\Phi} = \{\Phi_\tau\}$ between $(\tilde{G},\tilde{H})$ and
  the constant deformation of $(G,H)$. Hence, by
  Remark \ref{rmk:contact_jacobi}, there
  exist a diffeotopy $\{\phi_{\tau}^{-1}\}$ of $M$ and a smooth
  1-parameter family $\{a_\tau\}$ of positive functions on $M$ with
  $a_0 \equiv 1$, such that $(\phi^{-1}_\tau,a_\tau)$
  is a Jacobi isomorphism between $(M,\pi)$ and
  $(M,\pi_\tau)$ for all $\tau$. Therefore, by equation \eqref{eq:conf_iso},
  $$ (\phi_\tau^{-1})_*(a_\tau \pi) = \pi_\tau \qquad \text{and}
  \qquad \pi^{\sharp}(da_\tau) = 0.$$
  \noindent
  The result follows by applying $(\phi_\tau)_*$ to both sides of the
  first equation and by observing that the second is precisely the
  condition that $a_\tau$ be a Casimir of $\pi$.
\end{proof}

\section{Three families of examples}\label{sec:three_examples}
In this section we give three families of examples of compact contact
groupoids and of the Jacobi bundles that they induce. In each case, we
observe how Theorems
\ref{thm:gray_Jacobi} and \ref{thm:gray_poisson} can be applied (see Remarks \ref{rmk:deformation_zero}, \ref{rmk:ionut}
and \ref{rmk:defo_pmcts}).


\subsection{First jet bundles and integral projective lattices}\label{sec:first-jet-bundles}
Let $p: L \to M$ be a line bundle. The 
{\bf first jet bundle of $\boldsymbol{L}$} is the vector 
bundle $\mathrm{pr} : J^1L \to M$, where, for $x \in M$,
$$ J^1L|_x := \{j^1_xu \mid u \in \Gamma_{\mathrm{loc}}(L) \}. $$
\noindent
There is a canonical contact structure $H_{\mathrm{can}}$ on 
$J^1L$ that is the kernel of the
{\bf Cartan contact form} 
$\alpha_{\mathrm{can}} \in \Omega^1(J^1L; \mathrm{pr}^*L)$. The latter
is given by
$$ \alpha_{\mathrm{can},j^1_xu}:= d_{j^1_xu}(\mathrm{ev} - u \circ \mathrm{pr}) : T_{j^1_xu} J^1L \to L_x, $$
\noindent
where $\mathrm{ev} : J^1L \to L$ is the map $j^1_x u \mapsto u(x)$, and $\ker dp$ is 
canonically identified with $p^*L$. In fact, $(J^1L,H_{\mathrm{can}})$
is a contact groupoid. This is because any vector bundle $E \to M$ is a Lie
groupoid and, in this case, a multiplicative distribution is a vector
subbundle of $TE \to TM$ with base $TM$. Since $J^1L$ is a vector bundle, the contact groupoid
$(J^1L,H_{\mathrm{can}})$ integrates $(M,L,0)$.

While $(J^1L,H_{\mathrm{can}})$ is never compact, in some
special cases there are {\em quotients} that 
are compact (if $M$ is compact), and still integrate $(M,L,0)$. Before
defining what we quotient $J^1L$ by, we recall that 
the Cartan contact form 
$\alpha_{\mathrm{can}}$ enjoys the following property. A section $\vartheta
\in \Gamma_{\mathrm{loc}}(J^1L)$ satisfies
$\vartheta^*\alpha_{\mathrm{can}} = 0 $ if and only if it is {\bf
  holonomic}, i.e., there exists $u \in \Gamma_{\mathrm{loc}}(L)$ such that
$\vartheta = j^1u$.

\begin{defn}[Definition $4.10$ in \cite{sal_sepe}]\label{defn:z-proj-struct}
  An {\bf integral projective lattice on
  $\boldsymbol{L\to M}$} is a full rank lattice $\Sigma\subset J^1L$ such 
  that any local section of $\Sigma \to M$ is holonomic.
\end{defn}

\begin{Rmk}\label{rmk:integral_projective_structures}
Integral projective lattices arise naturally when considering 
{\em integral projective structures}
(see \cite[Definition $4.12$]{sal_sepe}), which are special cases of real projective structures 
(see, e.g., \cite{goldman} and references therein). There is a $1-1$ correspondence between integral projective 
lattices on line bundles and integral projective structures (see \cite[Proposition $4.14$]{sal_sepe}). Moreover, 
proper contact groupoids 
and integral projective lattices/structures are related in a way that
is analogous to the relation between proper symplectic groupoids and 
integral affine structures (see 
\cite[Sections $4.2$ and $4.3$]{PMCT} and 
\cite[Theorem $1.0.1$]{PMCT2}). We will explore this relation 
in a separate paper.
\end{Rmk}

If $M$ is compact, an integral projective lattice $\Sigma$ on $L \to M$
yields a compact integration of $(M,L,0)$ as
follows. The fiberwise action 
of $\Sigma$ on $J^1L$ by translations is free and proper. Its
orbit space is a bundle of tori $\mathsf{pr}: J^1L/\Sigma \to M$ 
and the quotient map $Q : J^1L \to J^1L/\Sigma$ is a covering 
map satisfying $\mathsf{pr} \circ Q = \mathrm{pr}$. Moreover, 
since (local) sections of $\Sigma \to L$ are holonomic, the 
action is by contactomorphisms of $(J^1L,H_{\mathrm{can}})$, 
as the action preserves the Cartan contact form 
$\alpha_{\mathrm{can}}$. Hence, there is a contact structure 
$\bar{H}_{\mathrm{can}}$ on $J^1L/\Sigma$ satisfying 
$dQ (H_{\mathrm{can}}) = \bar{H}_{\mathrm{can}}$. Since the 
fiberwise action of $\Sigma$ is by Lie groupoid isomorphisms, 
$(J^1L/\Sigma, \bar{H}_{\mathrm{can}})$ is a contact groupoid. 
Moreover, since $J^1L/\Sigma \rightrightarrows M$ is a bundle of tori,
$J^1L/\Sigma$ is proper and $(J^1L/\Sigma,
\bar{H}_{\mathrm{can}})$ integrates $(M,L,0)$. Finally, compactness
of $M$ implies that $J^1L/\Sigma$ is compact. 

\begin{Rmk}\label{rmk:deformation_zero}
  If $M$ is compact, the existence of an integral projective lattice on $L \to M$ allows
  us to use Theorem \ref{thm:gray_Jacobi}: any compact integrable
  deformation of $(M,L,0)$ is trivial. While such deformations are
  probably very restrictive (as they come from deformations of contact
  groupoids), the above is a partial stability result for
  deformations of the
  {\em zero} Jacobi bracket in the presence of an integral
  projective lattice.
\end{Rmk}

\subsection{The (oriented) projectivization of the cotangent bundle to a compact Lie group}\label{sec:projectivisation_dual_Lie_algebra}
The {\bf projectivization} of a cotangent bundle $T^*M$ is $ \mathbb{P}(T^*M) := (T^*M \smallsetminus \{0\})/\Rr^*$,
\noindent
where 
$\Rr^*:= \Rr \smallsetminus \{0\}$ acts by fiberwise 
scalar multiplication. There is a well-known contact
structure\footnote{By an abuse of notation we denote the canonical
  contact structures on $J^1L$ and on $\mathbb{P}(T^*M)$ with the same
  symbol. Seeing as it should be clear from the context what the
  ambient manifold is, we trust that this should not cause confusion.}
$H_{\mathrm{can}}$ on
$\mathbb{P}(T^*M)$ that comes from identifying $\mathbb{P}(T^*M)$ with
the manifold of contact elements of $M$ (see \cite[Appendix 4, Section
D]{arnold}). Explicitly, if $p : \mathbb{P}(T^*M) \to M$ denotes
projection, then, for all $\eta \in T^*M \smallsetminus \{0\}$, $ H_{\mathrm{can},[\eta]} = (dp)^{-1}(\ker \eta)$.
Similarly, it is possible to define  a contact structure
$H_{\mathrm{can}}^+$ on the {\bf oriented projectivization} of $T^*M$,
$ \mathbb{S}(T^*M):= (T^*M \smallsetminus \{0\})/\Rr^+ $. Moreover,
$(\mathbb{S}(T^*M), H_{\mathrm{can}}^+)$ is co-orientable. One way 
to construct a contact form for
$(\mathbb{S}(T^*M), H^+_{\mathrm{can}})$ is by choosing a 
Riemannian metric $\mathsf{g}$ on $M$. This identifies 
$\mathbb{S}(T^*M)$ with the unit sphere bundle $U(T^*M)$ with 
respect to $\mathsf{g}$. Under this identification, 
$H^+_{\mathrm{can}} = \ker \lambda_{\mathrm{can}}|_{U(T^*M)}$, where
$\lambda_{\mathrm{can}}$ is the Liouville 1-form. \\

There is a well-known multiplicative analog of the above construction (see, e.g., \cite[Example $2.3$]{zamb_zhu}, 
\cite[Example $7.11$]{crainic_zhu} or 
\cite[Example $3.8$]{dual_pairs}). Let $G$ be a Lie group with 
Lie algebra $\gg$. The cotangent bundle $T^*G$ can be seen as 
a Lie groupoid over $\gg^*$: once $T^*G$ is trivialized using right translations, this is the action Lie 
groupoid of the coadjoint action $\gg^* \curvearrowleft G$ 
(see \cite[Example $1.15$]{crainic_fernandes_book} for details 
on action Lie groupoids). The projectivization 
$\mathbb{P}(T^*G)$ (respectively $\mathbb{S}(T^*G)$) 
is a Lie 
groupoid over the projectivization $\mathbb{P}(\gg^*)$ of $\gg^*$ (respectively
the oriented projectivization 
$\mathbb{S}(\gg^*)$ of $\gg^*$), and, using the above identification, 
it is the 
action Lie groupoid associated to the induced action 
$\mathbb{P}(\gg^*) \curvearrowleft G$ (respectively 
$\mathbb{S}(\gg^*) \curvearrowleft G$). Moreover, the contact 
structure $H_{\mathrm{can}}$ (respectively 
$H^+_{\mathrm{can}}$) is 
multiplicative, so that $(\mathbb{P}(T^*G),H_{\mathrm{can}})$ 
(respectively $(\mathbb{S}(T^*G),H^+_{\mathrm{can}})$) is a 
contact groupoid. These Lie groupoids are compact exactly if $G$ is
compact. In this case, it is possible to choose 
a bi-invariant Riemannian metric on $\gg^*$. Using this metric to
identify $\mathbb{S}(T^*G)$ with $U(T^*G)$, we have that
$(\lambda_{\mathrm{can}}|_ {U(T^*G)}, 1)$ is a 
multiplicative contact form for 
$(U(T^*G),H^+_{\mathrm{can}})$ (see 
\cite[Example $2.3$]{zamb_zhu}).

The contact groupoids $(\mathbb{P}(T^*G),H_{\mathrm{can}})$ and
$(\mathbb{S}(T^*G),H^+_{\mathrm{can}})$ integrate the following
families of Jacobi bundles (see \cite[Example 3.8]{dual_pairs}). Let 
$\pi_\mathrm{lin} \in \XX^2(\gg^*)$ denote the 
Kirillov-Kostant-Souriau Poisson bivector on $\gg^*$, i.e., 
for $\xi \in \gg^*$ and $f,g \in C^{\infty}(\gg^*),$
\begin{equation}\label{eqn:kks}
  (\pi_\mathrm{lin})_\xi(d_\xi f,d_{\xi}g):=\langle \xi, [d_\xi f, d_\xi g] \rangle,  
\end{equation}
\noindent
where $d_\xi f, d_\xi g$ are identified with elements of 
$(\gg^*)^* \cong \gg$, $[\cdot,\cdot]$ is the Lie bracket on 
$\gg$, and $\langle \cdot,\cdot \rangle$ is the standard 
pairing between $\gg^*$ and $\gg$. By equation
\eqref{eqn:kks}, the subspace of homogeneous functions of 
degree one on $\gg^* \smallsetminus \{0\}$ is a Lie subalgebra 
of $C^{\infty}(\gg^*)$. Identifying such functions with 
sections of $\mathrm{O}(1) \to \mathbb{P}(\gg^*)$, i.e., the dual of
the tautological line bundle, we obtain a Jacobi bundle 
$(\mathbb{P}(\gg^*),\mathrm{O}(1),\{\cdot,\cdot\})$
(for further details, see 
\cite[Example $2.10$ and Appendix B]{dual_pairs}, and 
\cite[Example $2.10$]{sal_sepe}). Analogously, if $q :
\mathbb{S}(\gg^*) \to \mathbb{P}(\gg^*)$ is the standard double cover,
we obtain a Jacobi bundle $(\mathbb{S}(\gg^*),
q^*\mathrm{O}(1),\{\cdot,\cdot\}_{\mathbb{S}})$. A choice of inner product on
$\mathfrak{g}^*$ identifies $\mathbb{S}(\gg^*)$ with the unit sphere
$U(\gg^*)$ and determines a Jacobi pair for $(\mathbb{S}(\gg^*),
q^*\mathrm{O}(1),\{\cdot,\cdot\}_{\mathbb{S}})$. If, in addition, $G$ is compact, the above inner product can
be chosen to be bi-invariant and the resulting Jacobi pair is of the
form $(\pi_{U(\gg^*)},0)$ for some Poisson bivector $\pi_{U(\gg^*)}$
on $U(\gg^*) \cong \mathbb{S}(\gg^*)$.

\begin{Rmk}\label{rmk:ionut}
  If $G$ is a compact Lie group, we can use Theorem \ref{thm:gray_Jacobi}: any compact integrable deformation of $(\mathbb{P}(\gg^*),
  O(1).\{\cdot,\cdot\})$ (respectively $(\mathbb{S}(\gg^*),
  q^*\mathrm{O}(1),\{\cdot,\cdot\}_{\mathbb{S}})$) is trivial. In fact,
  upon choosing a bi-invariant inner product on $\gg^*$, we can use
  Theorem \ref{thm:gray_poisson}: any compact contact integrable 
  deformation $\{(U(\gg^*),\pi_\tau)\}$ of $(U(\gg^*),
  \pi_{U(\gg^*)})$ satisfies equation \eqref{eq:11}. This last result should be
  compared with \cite[Part (a) of Theorem 1]{Ionut_spheres}. Theorem 1 
  in {\em loc. cit.} is much stronger as it provides a complete
  local description of the moduli space of Poisson structures near $(U(\gg^*),
  \pi_{U(\gg^*)})$, but requires deeper results (Nash-Moser
  techniques), and semisimplicity of $\gg$. 
\end{Rmk}

\subsection{Prequantization of compact symplectic groupoids}\label{sec:prequant_symp_gpd}

A {\bf prequantization} of a symplectic manifold 
$(S,\omega)$ is a principal $S^1$-bundle $p : N \to S$ such 
that there exists a principal connection 
$\alpha \in \Omega^1(N)$ satisfying $d\alpha = p^*\omega$. A 
necessary and sufficient condition for $(S,\omega)$ to admit 
a prequantization is that the cohomology class of $\omega$ be 
{\em integral}, i.e., it lies in the image of the homomorphism 
$H^2(S;\Zz) \to H^2(S;\Rr)$ (see, e.g., 
\cite[Theorems $7.2.4$ and $7.2.5$]{Geiges}). In this case, 
$H:= \ker \alpha$ is a contact structure on $N$. We abuse terminology
and refer to $(N,H)$ as a prequantization of $(S,\omega)$.

In order to discuss the multiplicative analog of the above
construction, we recall that a {\bf symplectic groupoid} is a Lie groupoid 
$S \rightrightarrows M$ endowed with a symplectic form 
$\omega \in \Omega^2(S)$ that is multiplicative with values 
in the trivial representation $M \times \Rr$ (see 
Definition~\ref{defn:multiplicative_distn}). In analogy with Theorem
\ref{LocToGlob}, a symplectic
groupoids induces a unique Poisson structure on its base so that the
target map is Poisson (see \cite[Theorem
1.1]{coste}). Following \cite[Definition $5.1$]{crainic_zhu}, we say that a 
{\bf prequantization} of a symplectic groupoid $(S,\omega)$ 
over $M$ is a Lie groupoid extension of $S$ by the trivial 
$S^1$-bundle over $M$,
\begin{equation*}\label{eq:prequant}
1 \to M \times S^1 \to G \stackrel{p}{\to} S \to 1,
\end{equation*} 
\noindent
such that $p : G \to S$ is a prequantization of $(S,\omega)$ 
with the property that the connection $1$-form 
$\alpha \in \Omega^1(G)$ is multiplicative with values in the 
trivial representation $M \times \Rr$. Setting 
$H:= \ker \alpha$, we have that $(G,H)$ is a co-orientable contact groupoid 
(see \cite[Theorem $3.1$ and Lemma $3.2$]{Weinstein_Xu} and 
the remark after \cite[Definition $5.1$]{crainic_zhu}). Moreover,
$(\alpha,1)$ is a contact form for $(G,H)$. As above, we
refer to $(G,H)$ as a prequantization of $(S,\omega)$.

\begin{Rmk}\label{rmk:prequant_nec_suf}
  In general, given a symplectic
  groupoid $(S,\omega)$, integrality of $\omega$ is only a necessary
  condition for the existence of a prequantization (see \cite[Theorem $3.1$]{Weinstein_Xu} for 
  necessary and sufficient conditions). However, if $S$ is 
  Hausdorff and source simply connected, then integrality of 
  $\omega$ is sufficient (see \cite[Theorem $3$]{crainic_zhu}).
\end{Rmk}

Let $(S,\omega)$ be a symplectic groupoid integrating a Poisson
manifold $(M,\pi)$ and let $(G,H = \ker \alpha)$ be a prequantization
of $(S,\omega)$. Then
$(G,\alpha,1)$ also integrates the Poisson manifold
$(M,\pi)$. Moreover, $G$ is compact if and only if $S$ is compact. 

\begin{Rmk}\label{rmk:defo_pmcts}
  If $(S,\omega)$ is a prequantizable compact symplectic groupoid 
  integrating a Poisson manifold $(M,\pi)$, we can apply Theorem
  \ref{thm:gray_poisson}: any compact contact integrable deformation
  $\{(M,\pi_\tau)\}$ of $(M,\pi)$ satisfies equation
  \eqref{eq:11}. In particular, by Remark \ref{rmk:prequant_nec_suf},
  this result applies to Poisson
  manifolds admitting a compact s-simply connected symplectic
  integration with integral symplectic form. These are examples of {\em Poisson manifolds of compact type} (see
  \cite{PMCT,PMCT2}).
\end{Rmk}

\appendix
\section{$\sigma$-Multiplicative forms}\label{sec:sigma-mult-forms}


Let $\sigma : G \to \{ \pm 1\}$ be a Lie groupoid homomorphism and let $E
\rightrightarrows V$ be a VB-groupoid over $G \rightrightarrows M$
(see \cite{mack} for the definition of 
VB-groupoids). The Lie groupoid homomorphism $\sigma$ can be used to `twist' the VB-groupoid
structure of $E \rightrightarrows V$ over $G \rightrightarrows M$ as
follows. For $e\in E_g$ and 
$(e_1,e_2)\in E_g^\sigma\times_{V_x}E_h^\sigma$,  the $\sigma$-twisted structure maps are
\begin{equation}
  \label{eq:27}
  \begin{split}
     &s_\sigma(e) := \sigma(g)s(e), \qquad t_\sigma(e) := t(e), \qquad u_\sigma := u,\\
   &m_\sigma(e_1,e_2) := m(e_1,\sigma(g)e_2), \qquad i_\sigma(e):=\sigma(g)i(e).
  \end{split}
\end{equation}

Since $E \rightrightarrows V$ is a VB-groupoid over $\gpd$ and since
$\sigma: G \to \{ \pm 1\}$ is a Lie groupoid homomorphism, the following
result holds at once.

\begin{lemma}\label{lemma:aux_gpd}
  The structure maps of equation \eqref{eq:27} define a VB-groupoid
  $E^\sigma\rightrightarrows V$ over $\gpd$. Moreover, if $\Phi:E_1\to
  E_2$ is a fiberwise linear groupoid homomorphism covering the identity on
  $\gpd$, then $\Phi:E_1^\sigma\to E_2^\sigma$ is also a fiberwise
  linear groupoid homomorphism covering the identity. 
\end{lemma}

Following \cite{bu_ca_or,Mack_Xu}, we use Lemma~\ref{lemma:aux_gpd} to
express $\sigma$-multiplicativity of a form in terms of Lie groupoid 
homomorphisms replacing the usual cotangent groupoid by 
$(T^*G)^\sigma\rightrightarrows A^*_G$ (see \cite{mack} for 
the definition of the cotangent groupoid $T^*G$). The following result
holds by the arguments in \cite[Lemma $3.6$]{bu_ca_or} with the obvious 
adaptations.

\begin{lemma}\label{lemma:sig-mult} Let $\sigma:G\to \{ \pm 1\}$ be a groupoid homomorphism, then
\begin{enumerate}[leftmargin=*]
\item a 1-form $\beta\in\Omega^1(G)$ is $\sigma$-multiplicative if and
  only if $\beta$ defines a Lie groupoid homomorphism $G\to(T^*G)^\sigma$, and 
\item a 2-form $\omega\in\Omega^2(G)$ is $\sigma$-multiplicative if
  and only if $\omega^\flat$ defines a Lie groupoid homomorphism
$$\xymatrix{TG \ar@<0.5ex>[d]\ar@<-0.5ex>[d]\ar[r]^{\omega^\flat} & (T^*G)^\sigma \ar@<0.5ex>[d]\ar@<-0.5ex>[d] \\
TM \ar[r] & A_G^*.}$$
\end{enumerate}
\end{lemma}


\end{document}